\documentclass[12pt]{amsart}
\usepackage{amsmath,amsfonts,amssymb,amsthm}
\usepackage{anysize}
\marginsize{2cm}{2cm}{2cm}{2cm}
\usepackage{graphicx}
\usepackage{color}
\usepackage{comment}
\usepackage{float}

\usepackage[pageref]{backref}
\usepackage[colorlinks=true, linkcolor=blue, citecolor=red, urlcolor=blue, backref=page]{hyperref}
\renewcommand*{\backrefalt}[4]{%
  \ifcase #1 %
    (Not cited.)%
  \or
    (Cited on page~\hyperlink{page.#2}{#2}.)%
  \else
    (Cited on pages~\hyperlink{page.#2}{#2}.)%
  \fi
}

\usepackage{enumerate}
\usepackage{caption}
\usepackage{longtable}
\usepackage{multirow}
\usepackage{lscape} 
\usepackage{array}
\usepackage{multicol}

\usepackage{setspace}
\usepackage{tikz}
\usetikzlibrary{graphs}
\usetikzlibrary{arrows}
\usetikzlibrary{calc,fit,shapes.geometric}

\def\z{\mathfrak{z}}

\def\g{\mathfrak{g}}
\def\h{\mathfrak{h}}
\def\n{\mathfrak{n}}

\def\a{\mathfrak{a}}

\def\C{\mathbb{C}}
\def\R{\mathbb{R}}

\def\N{\mathbb{N}}

\def\G{\mathbb G}
\def\A{\mathbb A}
\def\B{\mathbb B}

\def\alt{\raise1pt\hbox{$\bigwedge$}}
\def\pint{\langle \cdotp,\cdotp \rangle }

\theoremstyle{plain}
\newtheorem{theorem}{\bf Theorem}[section]
\newtheorem{corollary}[theorem]{\bf Corollary}
\newtheorem{proposition}[theorem]{\bf Proposition}
\newtheorem{lemma}[theorem]{\bf Lemma}

\theoremstyle{definition}
\newtheorem{definition}[theorem]{\bf Definition}
\newtheorem{example}[theorem]{\bf Example}

\theoremstyle{remark}
\newtheorem{remark}[theorem]{\bf Remark}

\newcommand{\ri}{{\rm (i)}}
\newcommand{\rii}{{\rm (ii)}}
\newcommand{\riii}{{\rm (iii)}}
\newcommand{\riv}{{\rm (iv)}}

\title[Complex structures on Lie algebras arising from graphs]{Complex structures on 2-step nilpotent Lie algebras arising from graphs}

\author{Adrián Andrada}
\email{adrian.andrada@unc.edu.ar}
\author{Sonia Vera}
\email{svera@unc.edu.ar}

\date{}
\address{FAMAF, Universidad Nacional de C\'ordoba and CIEM-CONICET, Av. Medina Allende s/n, Ciudad Universitaria, X5000HUA C\'ordoba, Argentina}


\begin{document}

\begin{abstract}
This work investigates the existence of complex structures on 2-step nilpotent Lie algebras arising from finite graphs. We introduce the notion of adapted complex structure, namely a complex structure that maps vertices and edges of the graph to vertices and edges, and we analyze in depth the restrictions imposed by the integrability condition. We completely characterize the graphs that admit abelian adapted complex structures, showing that they belong to a small family of graphs that we call basic. We prove that any graph endowed with an adapted complex structure $J$ contains a unique $J$-invariant spanning basic subgraph, and conversely, that every such graph can be constructed through a systematic expansion procedure starting from a basic graph. We also explore geometric and combinatorial consequences, including the existence of special Hermitian metrics as well as other graph-theoretic properties.
\end{abstract}

\maketitle

\tableofcontents

\section{Introduction}

A complex structure on a real Lie algebra $\g$ is a linear map $J\colon \g \to \g$ such that $J^2 = -\operatorname{Id}$ and the eigenspaces of $J$ in the complexification $\g^\C$ are complex Lie subalgebras. 
The study of complex structures on nilpotent Lie algebras is an important topic in differential and Hermitian geometry. In his influential work \cite{Sal}, S.\ Salamon established several structural results for nilpotent Lie algebras that admit complex structures and provided the classification of such Lie algebras in dimension 6. Many other authors have studied complex structures on nilpotent Lie algebras from different points of view: for instance, the nilpotency of the complex structure itself (see, among others, \cite{CFGU, GZZ, LU, LUV-IJM, LUV-TAMS, LUV-JA, Rol}), or the existence of compatible Hermitian metrics with special properties (see, among others, \cite{AV, AN, BG, EFV, EFV1, FPS, ST, UV}). Recently, a study of 2-step nilpotent Lie algebras equipped with complex structures appeared in \cite{B}.

In parallel, a different line of research has investigated the construction of Lie algebras from combinatorial data, particularly finite graphs. The foundational work of Dani and Mainkar \cite{DaM} introduced a systematic method to associate a $k$-step nilpotent Lie algebra to any finite simple graph, leading to a fertile interplay between graph theory and the structure theory of nilpotent Lie algebras. These graph Lie algebras have been studied further from algebraic and geometric perspectives, especially for $k=2$, revealing how the combinatorics of the underlying graph influence algebraic and geometric properties (see, for instance, \cite{AD, DW1, DW2, Fa, LW, Ov, PT}).

This paper lies at the intersection of these two threads. Our goal is to study the existence of complex structures on 2-step nilpotent Lie algebras arising from graphs. Specifically, we investigate the existence of complex structures that are compatible with the starting graph; indeed, we consider complex structures that send a vertex or an edge to another vertex or edge and we call them \textit{adapted}. We simply say that the graph admits an adapted complex structure. In Section \ref{sec:definition} we provide the precise definition and exhibit some basic low-dimensional examples, as well as some infinite families of graphs that admit adapted complex structures (namely, the complete bipartite graphs $K_{2n,2m}$). The edges of the graph that connect a pair of vertices $v, w$ with $w=Jv$ play a fundamental role in this article, and we call them \textit{distinguished}. In Section \ref{sec:integrability} we thoroughly investigate the integrability condition for adapted complex structures and obtain several important results. For instance, if $J$ is an adapted complex structure on the graph $\G$, the image by $J$ of an isolated vertex is either a distinguished edge or an even vertex (Corollary \ref{cor: isolated}). Another important result is the fact that the image of a non-distinguished edge $a$ by $J$ is another non-distinguished edge adjacent to $a$ (Proposition \ref{prop: endpoint}). As consequences of this result, we have Corollary \ref{cor: same parity} which states that if $v,w$ are vertices with $w=Jv$ then $\deg(v)=\deg(w)$, and also Corollary \ref{cor:odd}, where we prove that the set of odd vertices is $J$-invariant. Another consequence is the fact that if both vertices $v$ and $w=Jv$ are odd, then these vertices cannot be far apart, in fact, $d(v,w)\leq 2$; this establishes an obstruction for the existence of adapted complex structures (see Corollary \ref{cor: vertices impares lejos}). On the other hand, if $v$ and $w=Jv$ are even vertices, then there is no upper bound on their distance, as Example \ref{ex: peine} shows.

Section \ref{sec:construction} contains the most important results in this article. First, we characterize the graphs that admit adapted complex structures $J$ that satisfy the additional condition of being \textit{abelian}, that is, $[Jx,Jy]=[x,y]$ for any $x,y$ in the corresponding Lie algebra. It turns out that the class of graphs that admit abelian adapted complex structures is very narrow, since they are the union of three possible graphs $\A_1,\, \A_2$ and $\A_3$ described in Example \ref{ex: first examples}, where each of them admits an abelian adapted complex structure. The Lie algebras associated with these graphs are, respectively, $\R^2$, $\h_3\times \R$ and $\h_3\times \h_3$, where $\h_3$ denotes the 3-dimensional Heisenberg Lie algebra. We call such a graph \textit{basic}, and they happen to be fundamental in our analysis. Indeed, in Theorem \ref{thm: existe basico} we show that any graph equipped with an adapted complex structure $J$ has a unique $J$-invariant spanning subgraph which is basic. Conversely, we show a method to construct a graph with an adapted complex structure beginning with a basic graph and adding pairs of edges in a suitable way; in Theorem \ref{thm: expanded graph} we prove the integrability of this adapted complex structure. We call this process an \textit{expansion} of the basic graph. Combining these last two theorems, we see that \textit{any} graph equipped with an adapted complex structure is an expansion of a basic graph (Theorem \ref{thm:characterization}). We use these results to show that the $n$-cycle $C_n$ admits an adapted complex structure if and only if $n=3$ or $n=4$ (Corollary \ref{cor:cycle}). 

In Section \ref{sec:further} we present some results that can be deduced from the characterization of graphs with adapted complex structures in terms of basic graphs. For instance, regarding the step of nilpotency of an adapted complex structure (in the sense of \cite{CFGU}), it is known that such a structure can only be 2-step or 3-step nilpotent (see \cite{GZZ,B}), and in Proposition \ref{prop:3-step} we characterize the graphs equipped with a 3-step nilpotent adapted complex structures. We also show that complete graphs with $4n$ or $4n+3$ vertices admit adapted complex structures, and the same happens with the complete graphs of $4n+1$ or $4n+2$ vertices, after the addition of an isolated vertex for dimensional reasons (see Proposition \ref{prop:complete}). As a consequence, we see that the chromatic number of a graph admitting adapted complex structures can be arbitrarily large. Other graph-theoretical properties of graphs equipped with adapted complex structures, such as the girth and the existence of perfect matchings, are investigated in Propositions \ref{prop:girth} and \ref{prop:matching}, respectively. In particular, using the notion of blow-up of a graph, we show in Proposition \ref{prop: blow} that any graph gives rise to other graphs that admit adapted complex structures. 

In Section \ref{sec:trees}, we study the existence of adapted complex structures on forests, that is, graphs that do not have cycles. Using the characterization of graphs with an adapted complex structure given in the previous sections, we prove that a forest admits an adapted complex structure if and only if it has an even number of connected components and each connected component has at most one vertex with even degree.

Finally, in Section \ref{sec:metrics} we address the problem of finding special Hermitian metrics associated with adapted complex structures. In particular, we completely characterize the graphs with an adapted complex structure that admits a balanced metric where the vertices form an orthogonal set; this happens if and only if the graph has no distinguished edges (see Proposition \ref{prop: balanced}). On the other hand, regarding the existence of SKT (or pluriclosed) metrics, we assume a mild condition on the metric (see equation \eqref{eq: complement}), and we show in Proposition \ref{prop: E orthogonal} that if such a metric exists and the edges form an orthogonal set then the graph must be basic (and hence the adapted complex structure is abelian). Furthermore, we give a non-trivial example of a non-basic graph equipped with an adapted complex structure and a compatible SKT metric (see Example \ref{ex:SKT}).

In summary, in this article we provide an easy procedure to produce 2-step nilpotent Lie algebras equipped with complex structures, beginning with certain basic graphs, and we determine some algebraic and geometric properties of these complex structures. We also exhibit many examples of graphs that admit adapted complex structures, as well as graphs that do not admit any. Our approach combines tools from Lie theory, complex geometry, and combinatorics. Beyond contributing to the understanding of complex structures on nilpotent Lie algebras, our results shed new light on how discrete data encoded in graphs can give rise to geometric structures, reinforcing the bridge between combinatorics and differential geometry.

\section{Preliminaries}\label{prelim}

\subsection{Graphs}

Since we will be working with graphs throughout the paper, a brief review is in order. 

A graph is a pair $\G = (V,E)$ where:
\begin{itemize}
    \item[$\diamond$] $V$ is a finite set, whose elements are called vertices or nodes, and 
    \item[$\diamond$] $E\subseteq \{ \{u,v\} \mid u,v\in V, \, u\neq v\}$ is a set of 2-element subsets of $V$; the elements of $E$ are called edges. We will usually denote the edge $\{u,v\}$ by $\overline{uv} \, (=\overline{vu})$. 
\end{itemize}
If $\overline{uv}$ is an edge, we say that $u$ and $v$ are adjacent. Two edges are adjacent, or incident, if they share a common vertex. If $v\in V$, the neighborhood $N_v$ of $v$ is the set of all vertices adjacent to $v$.

The degree of a vertex $v$ in the graph $\G$ is the number of edges incident with $v$, and it is denoted by $\deg(v)$. That is, $\deg(v)=|N_v|$. If $\deg(v)=0$ then $v$ is called an \textit{isolated} vertex. Furthermore, a vertex $v$ will be called \textit{even} or \textit{odd} if $\deg(v)$ is even or odd, respectively. Given a graph $\G$, it is well known that
\begin{itemize}
    \item[$\diamond$] the sum of the degrees of all the vertices is twice the number of edges,
    \item[$\diamond$] the number of vertices of odd degree is even.
\end{itemize}

A \textit{subgraph} of a graph $\G(V,E)$ is a graph $\G_1(V_1,E_1)$ such that $V_1\subseteq V$ and $E_1\subseteq E$. If $V_1=V$ then $\G_1$ is called a \textit{spanning subgraph} of $\G$. A subgraph $\G_1$ is an \textit{induced subgraph} of $\G$ if two vertices of $\G_1$ are adjacent if and only if they are adjacent in $\G$. 

A \textit{path} of length $r$ from $u$ to $v$ in a graph $\G$ is a sequence of $r+1$ different vertices starting with $u$ and ending with $v$ such that consecutive vertices are adjacent. If there is a path between any two vertices in $\G$, then $\G$ is called \textit{connected}. A maximal connected induced subgraph of $\G$ is called a \textit{connected component} of $\G$. A \textit{cycle} is a path that begins and ends at the same vertex, without passing through any other vertex twice. 

A \textit{wedge} is a path of length 2; it consists of three vertices and two edges, with a central vertex $u$ connected to the other two $v$ and $w$. We denote it by $(u;v,w)$; we call $u$ the center of the wedge, and $v,w$ the endpoints of the wedge.

The \textit{distance} between two vertices $v,w$ in a graph is the number of edges in the shortest path between them, and it is denoted by $d(v,w)$. For instance, $N_v=\{w\in  V\mid d(v,w)=1\}$. If there is no path between two vertices, then, by convention, the distance is defined as $+\infty$.  

\medskip

Next, we define some special graphs. A graph $\G=\G(V,E)$ is
\begin{itemize}
\item[$\diamond$] \textit{empty} if $E=\emptyset$,
\item[$\diamond$] \textit{complete} if every vertex is adjacent to all other vertices. The complete graph with $n$ vertices is denoted by $K_n$,
\item[$\diamond$] \textit{bipartite} if $V=V_1\sqcup V_2$ (disjoint union) in such a way that every edge connects a vertex in $V_1$ to a vertex in $V_2$, that is, there are no edges connecting vertices within the same set. It is \textit{complete bipartite} if every vertex in $V_1$ is adjacent to every vertex in $V_2$; the complete bipartite graphs are usually denoted $K_{m,n}$, where $m=|V_1|$ and $n=|V_2|$,
\item[$\diamond$] a \textit{cycle graph} if it consists of a single cycle. The cycle graph with $n$ vertices is usually denoted by $C_n$,
\item[$\diamond$] a \textit{tree} if it is connected and does not contain any cycles. A graph whose connected components are all trees is called a \textit{forest}. Equivalently, a forest is a graph with no cycles.
\end{itemize}

The union of two graphs $\G_1=\G_1(V_1,E_1)$ and $\G_2=\G_2(V_2,E_2)$ is the graph whose set of vertices is $V_1\cup V_2$ and its set of edges is $E_1\cup E_2$. Finally, given a graph $\G$, we use the notation $\G_\ast$ to denote the union of $\G$ with an isolated vertex. 

\subsection{Complex structures on Lie algebras}

A complex structure on a real Lie algebra $\g$ is an endomorphism $J$ of $\g$ satisfying $J^2=-\operatorname{Id}$ and such that
\begin{equation}\label{eq:N}
N_J(x,y):=[x,y]+J([Jx,y]+[x,Jy])-[Jx,Jy]=0 \quad \text{for all } x,y\in \g.
\end{equation}
The $\g$-valued 2-form $N_J$ is called the Nijenhuis tensor associated to $J$; the vanishing of $N_J$ is also known as the \textit{integrability} of $J$. This condition can also be understood in terms of the complexified Lie algebra $\g^\C=\g\otimes_\R \C$. Indeed, it is well known that \eqref{eq:N} holds if and only if ${\g}^{1,0}$, the $i$-eigenspace of the $\C$-linear extension of $J$ to $\g^\C$, is a complex Lie subalgebra of $\g ^{\C}$. In this case $\g^{0,1}$, the $(-i)$-eigenspace of the $\C$-linear extension of $J$ to $\g^\C$, is also a complex Lie subalgebra of $\g^\C$ and the decomposition $\g^\C=\g^{1,0}\oplus \g^{0,1}$ holds. A pair $(\g, J)$ will denote a Lie algebra $\g$ equipped with a complex structure $J$. Yet another equivalence for the integrability condition \eqref{eq:N} is given in terms of complex forms on $\g^\C$: if we decompose $(\g^\C)^*=\g^{*1,0}\oplus \g^{*0,1}$, with $\g^{*1,0}$ the annihilator of $\g^{0,1}$ and $\g^{*0,1}$ the annihilator of $\g^{1,0}$, then \eqref{eq:N} holds if and only if
\begin{equation}\label{eq: diff}
 d(\alt^{1,0}\g^*)\subseteq \alt^{2,0} \g^{*}\oplus \alt^{1,1}\g^{*}, 
\end{equation}
where $\alt^{p,q} \g^*:=\alt^p \g^{*1,0}\otimes \alt^q \g^{*0,1}$ and $d$ denotes the $\C$-linear extension of the usual Chevalley-Eilenberg differential for $\g$. That is, $J$ is integrable if and only if  the $(0,2)$-component of the differential of any $(1,0)$-form vanishes.

A useful and well-known observation is that
\begin{equation} \label{eq: simetrias}
    N_J(Jx,Jy) = N_J(x,y), \quad N_J(Jx,y)  = -JN_J(x,y), \quad N_J(x,Jy)  = -JN_J(x,y),
\end{equation}
for any $x,y\in \g$. Thus, the vanishing of $N_J(x,y)$ implies the vanishing of $N_J(Jx,Jy)$, $N_J(Jx,y)$ and $N_J(x,Jy)$.

A complex structure $J$ on $\g$ is called \textit{abelian} if the following condition holds:
\begin{equation}\label{eq: abelian}
[Jx,Jy]=[x,y], \qquad  x,y \in \g. 
\end{equation} 
This is equivalent to $\g^{1,0}$ and $\g^{0,1}$ being abelian Lie subalgebras of $\g^\C$, and also equivalent to $d(\alt^{1,0}\g^*)\subseteq \alt^{1,1}\g^{*}$. If $\g$ admits an abelian complex structure then $\g$ is $2$-step solvable, that is, the commutator ideal $[\g,\g]$ is abelian (see \cite[Proposition 4.3]{ABDO}), and it is easily seen that the center of $\g$ is $J$-invariant. 

Two complex structures $J_1$ and $J_2$ on the Lie algebra $\g$ are said to be \textit{equivalent} if there exists an automorphism $\varphi$ of $\g$ satisfying $J_2\,  \varphi = \varphi\, J_1$. 

\medskip

We point out that the existence of a complex structure $J$ on a Lie algebra $\g$ has geometric consequences for any Lie group $G$ with $\operatorname{Lie}(G)=\g$. Indeed, if $L_g$ denotes the left-multiplication on $G$ by $g\in G$, then we may define $J_g\colon T_gG\to T_gG$ by $J_g=(d L_g)_e J (d L_g)_g^{-1}$, where we have identified $\g$ with $T_eG$. It is easily seen that the endomorphism of the tangent bundle $J\colon TG\to TG$ thus defined is an almost complex structure on $G$, that is, $J_g^2=-\operatorname{Id}_{T_gG}$ for all $g\in G$, and it is integrable, i.e. $N_J(X,Y)=0$ for any vector fields $X,Y$ on $G$, where $N_J$ is the $(1,2)$-tensor defined on $G$ in the same way as in \eqref{eq:N}. It follows from the remarkable Newlander-Nirenberg theorem \cite{NN} that then $G$ is a complex manifold such that, by construction, the left-multiplications $L_g$ are holomorphic maps. Furthermore, if $\Gamma$ is a co-compact discrete subgroup of $G$ then the compact manifold $\Gamma\backslash G$ admits an induced complex structure; in this way we may produce examples of compact complex manifolds. 

When $G$ is nilpotent and simply connected the compact quotient $\Gamma\backslash G$ is called a \textit{nilmanifold}. There is a well-known criterion for the existence of co-compact discrete subgroups (called \textit{lattices}) in simply connected nilpotent Lie groups, given by Malcev in \cite{Mal}: such a Lie group has a lattice if and only if its Lie algebra admits a basis such that the corresponding structure constants are all rational numbers.

\subsection{2-step nilpotent Lie algebras arising from graphs} 
We will be concerned with 2-step nilpotent Lie algebras arising from graphs. First, we recall that a Lie algebra $\g$ is 2-step nilpotent if its commutator ideal $[\g,\g]$ is contained in the center of $\g$. Next, we briefly recall how to associate a 2-step nilpotent Lie algebra with a graph, following a construction introduced in \cite{DaM}.

Let $\G=\G(V,E)$ be a graph with set of vertices $V$ and set of edges $E$. Let $W_0$ be the real vector space generated by $V$ and let $W_1$ be the subspace of $\alt^2 V$ generated by $E$, that is,
\[ W_1=\text{span}_\R \{ v\wedge w \mid \text{$v$ and $w$ are vertices joined by an edge}\}. \]
Set $\n_\G:= W_0\oplus W_1$ and define a Lie bracket on $\n_\G$ as follows:
\begin{align} \label{eq: corchete}
& \diamond \text{if } v,w\in V \text{ then } [v,w]=\begin{cases} v\wedge w, \quad  \text{ if there is an edge joining } v \text{ and } w,\\
              0, \qquad  \quad \, \text{otherwise},
             \end{cases} \\
& \diamond W_1 \text{ is contained in the center}. \nonumber
\end{align}

If $V=\{v_1,\ldots, v_n\}$ and there is an edge between $v_i$ and $v_j$, this edge will be denoted by $e_{i,j}$, and we will  identify $v_i\wedge v_j$ with $e_{i,j}$ when $i<j$. Then, the Lie bracket between two vertices $v_i$ and $v_j$ with $i < j$ is the edge that joins them, if it exists, and zero otherwise. As a consequence, we may consider $E$ as a basis of $W_1$.

Clearly, $\n_\G$ is abelian if $E$ is empty, and it is a 2-step nilpotent Lie algebra of dimension $|V|+|E|$ if $E$ is not empty. 
Note that the commutator ideal of $\n_\G$ is $W_1$, and the center of $\n_\G$ is spanned by $E$ together with the set of isolated vertices. If the graph $\G$ is the union of graphs $\G_1, \ldots, \G_r$ (not necessarily the connected components) then it is clear that $\n_\G$ is isomorphic to $\n_{\G_1}\times \cdots \times \n_{\G_r}$. Moreover, it was proved in \cite{Ma} that given two graphs $\G_1$ and $\G_2$, the Lie algebras $\n_{\G_1}$ and $\n_{\G_2}$ are isomorphic if and only if $\G_1$ and $\G_2$ are isomorphic.

\begin{example}
Let us consider the graph $\G$ in the following figure:
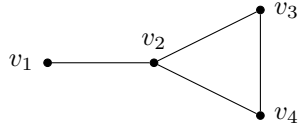
\[
\begin{tikzpicture}[scale=0.7]\footnotesize
 \tikzset{every loop/.style={looseness=-10}}
\tikzset{every loop/.style={looseness=30}}
\node (1) at (0,0) [draw,circle,inner sep=1pt,fill=black!100,label=180:{$v_1$}] {};
\node (2) at (2,0) [draw,circle,inner sep=1pt,fill=black!100,label=90:{$v_2$}] {};
\node (3) at (4,1) [draw,circle,inner sep=1pt,fill=black!100,label=360:{$v_3$}] {};
\node (4) at (4,-1) [draw,circle,inner sep=1pt,fill=black!100,label=360:{$v_4$}] {};

\path
(1) edge [-, black] node[below=2] {} (2)
(2) edge [-, black] node[above=2] {} (3)
(2) edge [-, black] node[below=2] {} (4)
(3) edge [-, black] node[right=3] {} (4);
\end{tikzpicture}  
\] 
The Lie bracket on $\n_\G$ is given by
\[  [v_1,v_2]=e_{1,2}, \; [v_2,v_3]=e_{2,3}, \; [v_2,v_4]=e_{2,4}, \; [v_3,v_4]=e_{3,4}.    \]
\end{example}

\begin{remark}
Note that the Lie bracket on $\n_\G$ satisfies the following very strong condition:
\begin{equation}\label{eq:property-graph}
 \text{if } v_i,v_j,v_k,v_l\in V \text{ satisfy } [v_i,v_j]=[v_k,v_l]\neq 0  \text{ then } i=k \text{ and } j=l.
\end{equation}
\end{remark}

\begin{remark}
The simply connected Lie group corresponding to a 2-step nilpotent Lie algebra arising from a graph has lattices. Indeed, in the basis $\mathcal{B}=V\cup E$ of $\n_\G$, with $\G=\G(V,E)$, the structure constants are $0, 1$ or $-1$, all rational numbers.  
\end{remark}

\begin{example}\label{heis}
Let us consider the $(2n+1)$-dimensional Heisenberg Lie algebra $\h_{2n+1}$, which has a basis $\{v_i,w_i\}_{i=1}^n\cup\{z\}$ with Lie bracket $[v_i,w_i]=z$ for $1\leq i\leq n$. Then it is clear that $\h_3$ arises from the graph with two vertices joined by one edge. Moreover, it follows from \eqref{eq:property-graph} that the $\h_{2n+1}$ does not arise from a graph if $n>1$.
\end{example}

\begin{remark}
Another way to describe the 2-step nilpotent Lie algebra associated to the graph $\G=\G(V,E)$ is as a quotient of the free 2-step nilpotent Lie algebra generated by the vertices, that is, $V\oplus \alt^2 V$ with Lie bracket $[v,w]=v\wedge w\in \alt^2 V$ for $v,w\in V$, and $\alt^2 V$ central. Indeed, if we set  
\[ R:=\text{span}_{\R}\{ v\wedge w \mid \text{$v$ and $w$ are vertices \emph{not} connected by an edge}\},  \]
then $R$ is an ideal of $V\oplus \alt^2 V$ and it is clear that $\n_\G=(V\oplus \alt^2 V)/R$.
\end{remark}

\section{Adapted complex structures: definition and first examples}\label{sec:definition}

In this section we introduce the main object of study of this article, namely, complex structures on 2-step nilpotent Lie algebras arising from graphs, which are compatible in a certain sense with these graphs.

\begin{definition}
Let $\G=\G(V,E)$ be a graph, and $\n_\G$ the associated $2$-step nilpotent Lie algebra. A complex structure $J$ on $\n_\G$ is called \textit{adapted} if for any $x\in V\cup E\subset \n_\G$ we have that either $Jx\in V\cup E$ or $-Jx\in V\cup E$.  In other words, $J$ sends a vertex or an edge to a vertex or an edge (or their negatives).
Any edge joining a pair of vertices $v, w$ with $w=Jv$ will be called \textit{distinguished}.

We will say that a graph $\G$ admits an adapted complex structure if $\n_\G$ carries an adapted complex structure, for short.
\end{definition}

\begin{remark}
It is clear from the definition that two distinguished edges are never adjacent.
\end{remark}

\begin{example}\label{ex: first examples}
(i) Let $\A_1$ denote the empty graph with two vertices $v_1$ and $v_2$, as shown in Figure \ref{fig: R2}. Then $\n_{\A_1}$ is the 2-dimensional abelian Lie algebra and it carries an adapted complex structure $J$ given by $Jv_1=v_2$ and $J^2=-\operatorname{Id}$.
\begin{figure}[h]
    \centering
    \begin{tikzpicture}\footnotesize
\node (1) at (-1,0) [draw,circle,inner sep=1pt,fill=black!100,label=90:{$v_1$}] {};
\node (2) at (0.5,0)  [draw,circle,inner sep=1pt,fill=black!100,label=90:{$v_2$}] {};
\end{tikzpicture}
    \caption{Graph $\A_1$}
    \label{fig: R2}
\end{figure} 

\smallskip

(ii) Consider the graph $\A_2$ given in Figure \ref{fig: h3R}. 
\begin{figure}[h]
    \centering
    \begin{tikzpicture}\footnotesize
\node (1) at (-1,0) [draw,circle,inner sep=1pt,fill=black!100,label=90:{$v_1$}] {};
\node (2) at (1,0)  [draw,circle,inner sep=1pt,fill=black!100,label=90:{$v_2$}] {};
\node (3) at (3,0) [draw,circle,inner sep=1pt,fill=black!100,label=90:{$v_3$}] {};

\path
(1) edge [-, black] node[above=-0] {} (2);

\end{tikzpicture}
    \caption{Graph $\A_2$}
    \label{fig: h3R}
\end{figure} 
The Lie bracket on $\n_{\A_2}$ is given simply by $[v_1,v_2]=e_{1,2}$, so that $\n_{\A_2}\cong \h_3\times \R$. It can be easily verified that the endomorphism $J$ of $\n_{\A_2}$ defined by $Jv_1=v_2, \, Jv_3=e_{1,2}$ and $J^2=-\operatorname{Id}$ is an adapted complex structure on $\n_{\A_2}$. Moreover, this complex structure is abelian.

\smallskip

(iii) Consider now the graph $\A_3$ given in Figure \ref{fig: h3R-2}.
The Lie bracket on $\n_{\A_3}$ is given by $[v_1,v_2]=e_{1,2}$ and $[v_3,v_4]=e_{3,4}$, hence we have an isomorphism $\n_{\A_3}\cong \h_3\times \h_3$. There is an adapted complex structure on $\n_{\A_3}$ given by $Jv_1=v_2, \, Jv_3=v_4, \, Je_{1,2}=e_{3,4}$ and $J^2=-\operatorname{Id}$. Again, it is readily verified that this complex structure is abelian.
\begin{figure}[H]
    \centering
    \begin{tikzpicture}\footnotesize
\node (1) at (-1,0) [draw,circle,inner sep=1pt,fill=black!100,label=90:{$v_1$}] {};
\node (2) at (1,0)  [draw,circle,inner sep=1pt,fill=black!100,label=90:{$v_2$}] {};
\node (3) at (3,0) [draw,circle,inner sep=1pt,fill=black!100,label=90:{$v_3$}] {};
\node (4) at (5,0) [draw,circle,inner sep=1pt,fill=black!100,label=90:{$v_4$}] {};

\path
(1) edge [-, black] node[above=-0] {} (2)
(3) edge [-, black] node[above=-0] {} (4);
\end{tikzpicture}
    \caption{Graph $\A_3$}
    \label{fig: h3R-2}
\end{figure} 
\end{example}

\begin{remark}
In \cite[Proposition 2.2]{ABD} it was shown that the Lie algebra $\n_{\A_2}\cong \h_3\times \R$ admits, up to equivalence, only one complex structure, which is abelian; therefore, it coincides with the adapted complex structure given in Example \ref{ex: first examples}$\rii$. On the other hand, it was proved in \cite[Theorem 3.3]{ABD} that the Lie algebra $\n_{\A_3}\cong \h_3\times \h_3$ admits a 1-dimensional family $\{J_s\}_{s\in\R}$, of abelian complex structures; however, it is easily verified that only $J_0$ is adapted to the graph $\A_3$. This complex structure $J_0$ coincides with the adapted complex structure given in Example \ref{ex: first examples}$\riii$.
\end{remark}

\smallskip

We next exhibit an infinite family of graphs that admit adapted complex structures, namely, some complete bipartite graphs.

\begin{proposition}\label{prop: bipartite}
Let $K_{2n,2m}$ be the complete bipartite graph on $2n$ and $2m$ vertices, with $n,m\in\N$. Then $K_{2n,2m}$ admits an adapted complex structure.
\end{proposition}

\begin{proof}
Let us write $K_{2n,2m}=V_1\sqcup V_2$ with $V_1=\{v_1,\ldots, v_{2n}\}$
and $V_2=\{ w_1,\ldots, w_{2m}\}$, and let us denote by $e_{i,j}$ the edge that joins the vertex $v_i$ with the vertex $w_j$. Therefore the Lie bracket on the associated Lie algebra $\n_{K_{2n,2m}}$ is given by
\[  [v_i,w_j]=e_{i,j}, \quad 1\leq i\leq 2n, \quad 1\leq j\leq 2m.  \]
We consider the almost complex structure on $\n_{K_{2n,2m}}$ given by
\[  Jv_{2i-1}= v_{2i}, \quad Jw_{2i-1}= w_{2i}, \quad Je_{i,2j-1}= e_{i,2j}, \quad J^2=-\operatorname{Id}. \]
We verify that the Nijenhuis tensor $N_J$ vanishes. Since the center $\z$ of $\n_{K_{2n,2m}}$, spanned by $\{e_{i,j}\}$, is $J$-invariant, we only need to verify that $N_J(v_{2i-1},v_{2j-1})=0$, $N_J(w_{2i-1},w_{2j-1})=0$ and $N_J(v_{2i-1},w_{2j-1})=0$. The first two are a consequence of the fact that both $\text{span}\{v_i\}$ and $\text{span}\{w_i\}$ are $J$-invariant abelian subalgebras of $\n_{K_{2n,2m}}$. Finally, we compute
\begin{align*}
    N_J(v_{2i-1},w_{2j-1}) & = [v_{2i-1},w_{2j-1}]+J([v_{2i},w_{2j-1}]+[v_{2i-1},w_{2j}])-[v_{2i},w_{2j}] \\
    & = e_{2i-1,2j-1}+J(e_{2i,2j-1}+e_{2i-1,2j})-e_{2i,2j} \\
    & = e_{2i-1,2j-1}+e_{2i,2j}-e_{2i-1,2j-1}-e_{2i,2j}\\
    & =0.
\end{align*}
Therefore, $J$ is integrable. 
\end{proof}

\begin{remark}\label{rem: bipartite}
There is another complex structure on $K_{2n,2m}$, defined analogously by:
\[  Jv_{2i-1}= v_{2i}, \quad Jw_{2i-1}= w_{2i}, \quad Je_{2i-1,j}= e_{2i,j}, \quad J^2=-\operatorname{Id}. \]
\end{remark}

\begin{example}\label{ex:G1G2}
We next exhibit some graphs equipped with adapted complex structures that are not abelian, as can be easily verified. 

(i) Consider the graph $\G_1$ given by the figure below:
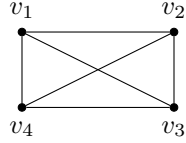
\begin{center}
\begin{tikzpicture}[scale=1]\footnotesize
    \node (1) at (0,0) [draw,circle,inner sep=1pt,fill=black!100,label=270:{$v_4$}] {};
    \node (2) at (2,0) [draw,circle,inner sep=1pt,fill=black!100,label=270:{$v_3$}] {};
    \node (3) at (2,1) [draw,circle,inner sep=1pt,fill=black!100,label=90:{$v_2$}] {};
    \node (4) at (0,1) [draw,circle,inner sep=1pt,fill=black!100,label=90:{$v_1$}] {};

    \path
    (1) edge  (2)
    (2) edge  (4)
    (3) edge  (1)
    (3) edge  (4)
    (2) edge  (3)
    (4) edge  (1);    
\end{tikzpicture}
\end{center}
Then it can be easily verified that $\n_{\G_1}$ carries an adapted complex structure $J_1$ given by
\[ J_1v_1=v_2, \quad J_1v_3=v_4, \quad J_1e_{1,2}=e_{3,4}, \quad J_1e_{1,3}=e_{1,4}, \quad J_1e_{2,3}=e_{2,4}. \]

\smallskip 

(ii) Consider now the graph $\G_2$ given by the figure below:
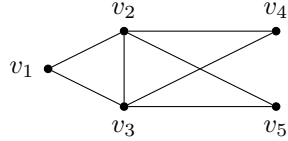
\begin{center}
\begin{tikzpicture}[scale=1]\footnotesize
    \node (1) at (0,0) [draw,circle,inner sep=1pt,fill=black!100,label=270:{$v_3$}] {};
    \node (2) at (2,0) [draw,circle,inner sep=1pt,fill=black!100,label=270:{$v_5$}] {};
    \node (3) at (2,1) [draw,circle,inner sep=1pt,fill=black!100,label=90:{$v_4$}] {};
    \node (4) at (0,1) [draw,circle,inner sep=1pt,fill=black!100,label=90:{$v_2$}] {};
    \node (5) at (-1,0.5) [draw,circle,inner sep=1pt,fill=black!120,label=180:{$v_1$}] {};
    \path
    (1) edge  (2)
    (2) edge  (4)
    (3) edge  (1)
    (3) edge  (4)
    (4) edge  (1)
    (5) edge  (1)
    (5) edge  (4);
\end{tikzpicture}
\end{center}
Then it can be easily verified that $\n_{\G_2}$ carries an adapted complex structure $J_2$ given by
\[ J_2v_1=e_{2,3}, \quad J_2v_2=v_3, \quad J_2v_4=v_5, \quad J_2e_{1,2}=e_{1,3}, \quad J_2e_{2,4}=e_{3,4}, \quad J_2e_{2,5}=e_{3,5}. \]
\end{example}

\section{Consequences of the integrability condition} \label{sec:integrability}

In this section, we derive the first important structural results on graphs that admit an adapted complex structure; they will be a direct consequence of the integrability condition, that is, the vanishing of the Nijenhuis tensor \eqref{eq:N}.

Let $\G$ be a graph equipped with an adapted complex structure $J$, and assume that $u,v$ are vertices of $\G$ with $[u,v]\neq 0$ (that is, they are joined by the edge $\overline{uv}$). It follows from \eqref{eq:N} that 
\begin{equation}\label{eq: integrable}
    0 \neq J[u,v]=[Ju,v]+[u,Jv]+J[Ju,Jv].
\end{equation} 
Therefore, at least one of the three terms on the right-hand side of \eqref{eq: integrable} does not vanish. Moreover, the definition of an adapted complex structure implies that $J[u,v]$ must be precisely equal to one of the terms on the right-hand side of \eqref{eq: integrable}, and the other two have to cancel out. Consequently, we have the following possibilities: 
\begin{enumerate}
    \item[(P1)]  If $J[u,v]=J[Ju,Jv]$ and $[Ju,v]+[u,Jv]=0$ then we must have $v=Ju$ or $u=Jv$, that is, the edge $\overline{uv}$ is distinguished. Note that $[Ju,v]=[u,Jv]=0$.
    \item[(P2)] If $J[u,v]=[Ju,v]$ and $[u,Jv]+J[Ju,Jv]=0$, then $\pm Ju$ is a non-isolated vertex, say $w\in V$, with $w$ connected to $v$, and we have two cases:
    \begin{enumerate}
        \item $[u,Jv]=J[Ju,Jv]=0$; in this case, either $\pm Jv$ is an edge or $\pm Jv$ is a vertex not connected neither to $u$ nor $w$.
        \item $J[u,Jv]=[Ju,Jv]\neq 0$; in this case, $\pm Jv$ is a vertex connected to both $u$ and $w$.
    \end{enumerate}
    \item[(P3)] If $J[u,v]=[u,Jv]$ and $[Ju,v]+J[Ju,Jv]=0$, then $\pm Jv$ is a non-isolated vertex, say $x\in V$, with $x$ connected to $u$, and we have two cases:
    \begin{enumerate}
        \item $[Ju,v]=J[Ju,Jv]=0$; in this case, either $\pm Ju$ is an edge or $\pm Ju$ is a vertex not connected neither to $v$ nor $x$.
        \item $J[Ju,v]=[Ju,Jv]\neq 0$; in this case, $\pm Ju$ is a vertex connected to both $v$ and $x$.
    \end{enumerate}
\end{enumerate}

\smallskip

As a consequence of (P1)--(P3) we have

\begin{corollary}\label{cor: distinguished}
The following statements are equivalent:
\begin{enumerate}
    \item[$\ri$] The edge $\overline{uv}$ is distinguished,
    \item[$\rii$] $[Ju,v]+[u,Jv]=0$,
    \item[$\riii$] $[Ju,v]=[u,Jv]=0$.
\end{enumerate}
\end{corollary}

Some first results on adapted complex structures are the following. In all these results $\G=\G(V,E)$ denotes a graph that admits an adapted complex structure $J$.

\begin{lemma}
If the edge $a\in E$ is distinguished and $Ja\in E$ then $Ja$ is also distinguished.
\end{lemma}

\begin{proof}
We may assume that $a=[v,Jv]$ and $Ja=[x,y]$. Then it follows from \eqref{eq:N} that
\begin{equation}\label{eq: distinguished}
-[v,Jv]=[Jx,y]+[x,Jy]+J[Jx,Jy].
\end{equation}
We first note from $\eqref{eq: distinguished}$ that both $Jx$ and $Jy$ are, up to sign, non-isolated vertices of $\G$. Indeed, if, say, $Jx$ were a central element in $\n_\G$, it would follow from \eqref{eq: distinguished} that $[Jv,v]=[x,Jy]\neq 0$. Due to \eqref{eq:property-graph} we would have $y=\pm x$ in $\n_\G$, which is impossible. 

Next, it follows from \eqref{eq: distinguished} and the definition of the Lie bracket on $\n_\G$ that $[Jv,v]$ has to be equal to one of the summands on the right-hand side of $\eqref{eq: distinguished}$, while the other two summands must cancel. But, as before, if we had $[Jv,v]=[Jx,y]$ or $[Jv,v]=[x,Jy]$ we would arrive at a contradiction. Therefore, the only possibility is $[Jv,v]=J[Jx,Jy]$ and $[Jx,y]+[x,Jy]=0$, and it follows from Corollary \ref{cor: distinguished} that $Ja=[x,y]$ is a distinguished edge. 
\end{proof}

As some kind of converse, we show the next result.

\begin{lemma}
Assume that $a$ and $b$ are two non-adjacent edges. If $Ja=b$ then both $a$ and $b$ are distinguished.
\end{lemma}

\begin{proof}
Suppose $a=\overline{uv}$ with $[u,v]=a$ and $b=\overline{xy}$ with $[x,y]=b$. The hypothesis means that $\{u,v\}\cap \{x,y\}=\emptyset$. On the other hand, from the integrability condition \eqref{eq: integrable} we get
\[ [x,y]=b=Ja=J[u,v]=[Ju,v]+[u,Jv]+J[Ju,Jv]. \]
If $[x,y]=[Ju,v]$ or $[x,y]=[u,Jv]$ then $x$ or $y$ must be equal to $u$ or $v$, a contradiction. Therefore, $[x,y]=J[Ju,Jv]$ and $[Ju,v]+[u,Jv]=0$, and this means that $a=[u,v]$ is distinguished. 

Interchanging the roles of $a$ and $b$ we obtain that $b$ is distinguished, as well.
\end{proof}

\begin{lemma}\label{lem:central}
If $w$ is a vertex in $\G$ such that $Jw$ is a central element in $\n_\G$, then the neighborhood $N_w$ is $J$-invariant, thus $w$ has even degree.

Moreover, if $Jw$ is, up to sign, an edge then this edge is distinguished.
\end{lemma}

\begin{proof}
If $\deg(w)=0$ there is nothing to prove. Assume now that $\deg(w)\geq 1$, and let $w_1\in N_w$, with $a_1:=\overline{ww_1}=\pm [w,w_1]$. Since $Jw$ is central in $\n_\G$, it follows from $N_J(w,w_1)=0$ that 
\[   [w,w_1]+J[w,Jw_1]=0, \]
that is, $J[w,w_1]=[w,Jw_1]\neq 0$. Hence, $w_2:=Jw_1$ is, up to sign, another vertex in $\G$ which is connected to $w$ through the edge $a_2:=\pm Ja_1$. If $\deg(w)=2$ we are done. If $\deg(w)>2$, given another vertex $w_3$ connected to $w$, with the same computation as above we can find yet another vertex $w_4$ connected to $w$ with $Jw_3=w_4$. Repeating this process, we will eventually find that the set of vertices joined to $w$ is of the form $\{w_1,Jw_1,\ldots, w_k,Jw_k\}$, thus it is $J$-invariant and $\deg(w)$ is even.

For the last statement, assume that $Jw$ is the edge that joins the vertices $u$ and $v$, with $Jw=[u,v]$. It follows from \eqref{eq:N} that
\[ -w=[Ju,v]+[u,Jv]+J[Ju,Jv].\]
Since the complex structure $J$ is adapted and $w$ is a vertex, we must have $[Ju,v]+[u,Jv]=0$, which implies that the edge joining $u$ and $v$ is distinguished, according to Corollary \ref{cor: distinguished}.
\end{proof}

\begin{remark}\label{rem:different-components}
The first part of Lemma \ref{lem:central} can be generalized in the following way: if $w$ is a vertex in $\G$ such that $Jw$ is another vertex of $\G$ and $w$ and $Jw$ are in different connected components of $\G$ then both $w$ and $Jw$ have even degree. The proof is exactly the same: if $w_1\in N_w$ then $[Jw,w_1]=0$ and conditions (P1)--(P3) above imply $[w,w_1]+J[w,Jw_1]=0$, and we continue the argument as in the proof of the lemma.
\end{remark}

\begin{corollary}\label{cor: isolated}
If $v$ is an isolated vertex in $\G$ then $Jv$ is, up to sign, either an even vertex or a distinguished edge.
\end{corollary}

The following result will be crucial in what follows.

\begin{proposition}\label{prop: endpoint}
Let $a\in E$ be a non-distinguished edge in $\G$, joining the vertices $x$ and $y$. Then $Ja$ is, up to sign, another non-distinguished edge in $\G$ that is adjacent to $a$; more precisely, at least one of $Jx$ or $Jy$ is, up to sign, a non-isolated vertex, and $Ja$ joins $x$ with $Jy$, or $y$ with $Jx$. 
\end{proposition}

\begin{proof}
We know that either $[Jx,y]\neq 0$ or $[x,Jy]\neq 0$, since otherwise the edge $a$ joining $x$ and $y$ would be distinguished (Corollary \ref{cor: distinguished}). 

Assume first that $Jx$ is a central element in $\n_\G$, and hence $[Jx,y]=0$. This implies that $[x,Jy]\neq 0$ and, therefore, $Jy$ is, up to sign, a non-isolated vertex of $\G$. Moreover, conditions (P1)--(P3) imply  
\[ J[x,y]=[x,Jy], \]
therefore $Ja$ is, up to sign, another edge with $x$ and $Jy$ as endpoints.

We may now assume that both $Jx$ and $Jy$ are, up to sign, non-isolated vertices in $\G$, with $Jx\neq \pm y$. If $[Jx,y]\neq 0$ and $[x,Jy]=0$ then conditions (P1)--(P3) imply that $[Jx,Jy]=0$ and $J[x,y]=[Jx,y]$. The latter says that $Ja$ is, up to sign, another edge with $y$ and $Jx$ as endpoints. If $[x,Jy]\neq 0$ and $[Jx,y]=0$ then, as before, we obtain that $[Jx,Jy]=0$ and $J[x,y]=[x,Jy]$. Therefore, in this case we have that $Ja$ is, up to sign, another edge with $x$ and $Jy$ as endpoints. Finally, if $[Jx,y]\neq 0$ and $[x,Jy]\neq 0$ then $[Jx,Jy]\neq 0$ and we have two possibilities:
\begin{enumerate}
    \item[$\ri$] $J[x,y]=[Jx,y]$ and $J[Jx,Jy]=-[x,Jy]$,
    \item[$\rii$] $J[x,y]=[x,Jy]$ and $J[Jx,Jy]=-[Jx,y]$.
\end{enumerate}
In (i), $Ja$ has endpoints $Jx$ and $y$, and in (ii), $Ja$ has endpoints $x$ and $Jy$.
\end{proof}

\begin{corollary} \label{cor: same parity}
If $v,w$ are vertices of $\G$ with $w=Jv$ then $\deg(v)$ and $\deg(w)$ have the same parity.
\end{corollary}

\begin{proof}
If $\deg(v)=0$, that is, $v$ is an isolated vertex, then it follows from Corollary \ref{cor: isolated} that $\deg(w)$ is even, hence the statement holds in this case. 

We may then assume that $\deg(v)\geq 1$ and $\deg(w)\geq 1$, so that the neighborhoods $N_v$ and $N_w$ are non-empty. We decompose each of these neighborhoods as disjoint unions $N_v=N_v^+\cup N_v^-$ and $N_w=N_w^+\cup N_w^-$, where
\begin{gather*}
    N_v^+=\{x\in N_v \mid \text{there is an edge between } x \text{ and } w\},\\
    N_v^-=\{x\in N_v \mid \text{there is no edge between } x \text{ and } w\}, 
\end{gather*}
and similarly,
\begin{gather*}
    N_w^+=\{x\in N_w \mid \text{there is an edge between } x \text{ and } v \},\\
    N_w^-=\{x\in N_w \mid \text{there is no edge between } x \text{ and } v \}.
\end{gather*}
Clearly, $N_v^+=N_w^+=N_v\cap N_w$. Therefore, it is enough to show that $|N_v^-|$ and $|N_w^-|$ have the same parity. More precisely, we will show that both $N_v^-$ and $N_w^-$ are $J$-invariant subsets of $\n_\G$, hence $|N_v^-|$ and $|N_w^-|$ are even.

Let $x\in N_v^-$. Since $a:=\overline{vx}$ is a non-distinguished edge, it follows from Proposition \ref{prop: endpoint} that $Ja$ is, up to sign, another non-distinguished edge, connecting either $v$ with $Jx$ or $x$ with $w$. The latter contradicts the fact that $x\in N_v^-$, so $Ja$ is, up to sign, the edge $\overline{vJx}$. Furthermore, conditions (P1)--(P3) imply that $[w,Jx]=0$, therefore, $Jx$ is not connected to $w$ and thus $Jx\in N_v^-$. We have proved that $N_v^-$ is $J$-invariant, and the same holds for $N_w^-$, so the proof is complete. 
\end{proof}

As a consequence of Lemma \ref{lem:central} and Corollary \ref{cor: same parity} we obtain:

\begin{corollary}\label{cor:odd}
If $w$ is an odd vertex in $\G$ then $Jw$ is another vertex in $\G$ with odd degree. In particular, the set of odd vertices of $\G$ is $J$-invariant. 
\end{corollary}

\medskip

We next prove a restriction for the existence of adapted complex structures.

\begin{proposition}\label{prop: vertices impares lejos}
Let $\G$ be a graph equipped with an adapted complex structure $J$. If $v$ and $w$ are vertices with $w=Jv$ and $d(v,w)\geq 3$ then both $v$ and $w$ are even vertices.
\end{proposition}

\begin{proof}
If $v$ is an isolated vertex then $\deg(v)=0$ and $\deg(w)$ is even due to Corollary \ref{cor: isolated}. 
Assume now that $v$ is non-isolated and let $x\in N_v$. According to Proposition \ref{prop: endpoint}, we have two possibilities: (i) $J(\overline{vx})=\pm\, \overline{xw}$, or (ii) $Jx$ is, up to sign, a vertex of $\G$, say $y$, and $J(\overline{vx})=\pm\, \overline{vy}$. In the former case, we would have the path $v,x,w$, but this contradicts the fact that $d(v,w)\geq 3$. Therefore, it must be $J(\overline{vx})=\pm\, \overline{vy}$. This happens for any vertex in $N_v$, and therefore $\deg(v)$ is even. 

The same analysis with $w$ instead of $v$ shows that $\deg(w)$ is also even.
\end{proof}

\begin{corollary}
With the same notation as in Proposition \ref{prop: vertices impares lejos}, if $v$ and $w$ with $w=Jv$ are odd vertices then $d(v,w)\leq 2$. In particular, they are in the same connected component of $\G$. 
\end{corollary}

\begin{corollary}\label{cor: vertices impares lejos}
Let $\G$ be a graph with the following property: there exists an odd vertex $v$ such that $d(v,w)\geq 3$ for any other odd vertex $w$. Then $\n_\G$ does not admit any adapted complex structure. 
\end{corollary}

\begin{example}
According to Corollary \ref{cor: vertices impares lejos}, the 2-step nilpotent Lie algebra associated to the following graph does not admit any adapted complex structure:
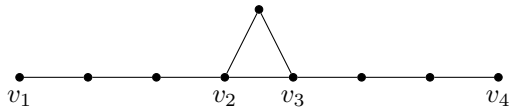
\[
\begin{tikzpicture}[scale=0.9]\footnotesize
\tikzset{every loop/.style={looseness=30}}
\node (1) at (0,1) [draw,circle,inner sep=1pt,fill=black!100] {};
\node (2) at (-3.5,0)  [draw,circle,inner sep=1pt,fill=black!100,label=270:{$v_1$}] {};
\node (3) at (-2.5,0) [draw,circle,inner sep=1pt,fill=black!100] {};
\node (4) at (-1.5,0) [draw,circle,inner sep=1pt,fill=black!100] {};
\node (5) at (-0.5,0) [draw,circle,inner sep=1pt,fill=black!100,label=270:{$v_2$}] {};
\node (6) at (0.5,0) [draw,circle,inner sep=1pt,fill=black!100,label=270:{$v_3$}] {};
\node (7) at (1.5,0)  [draw,circle,inner sep=1pt,fill=black!100] {};
\node (8) at (2.5,0) [draw,circle,inner sep=1pt,fill=black!100] {};
\node (9) at (3.5,0) [draw,circle,inner sep=1pt,fill=black!100,label=270:{$v_4$}] {};

\path
(1) edge [-, black] node[above=-0] {} (5)
(1) edge [-, black] node[above=-10] {} (6)
(2) edge [-, black] node[above=-10] {} (3)
(3) edge [-, black] node[above=0] {} (4)
(4) edge [-, black] node[above=-20] {} (5)
(5) edge [-, black] node[above=0] {} (6)
(6) edge [-, black] node[above=0] {} (7)
(7) edge [-, black] node[above=-20] {} (8)
(8) edge [-, black] node[above=-20] {} (9);
\end{tikzpicture} \] 
Indeed, the odd vertices are $v_1,v_2,v_3,v_4$, and the vertices satisfying the hypothesis of Corollary \ref{cor: vertices impares lejos} are $v_1$ or $v_4$. 
\end{example}

\begin{example}\label{ex: peine}
We exhibit a connected graph that admits an adapted complex structure $J$ such that there exists a pair of vertices $v,w$ with $Jv=w$ and $d(v,w)\geq 3$. 

For any $n\geq 3$, consider the following $2n$ points in the Euclidean plane $\R^2$:
\[ v_k=(k,1), \quad w_k=(k,0), \quad k=1,\ldots, n. \]
Let $F_n$ denote the graph with vertices $V=\{v_k,w_k\mid k=1,\ldots, n\}$ and edges:
\begin{itemize}
    \item $a_k=\overline{v_k v_{k+1}}$, $k=1,\ldots, n-1$, 
    \smallskip 
    \item $b_k=\overline{v_k w_k}$, $k=1,\ldots, n$,
    \smallskip
    \item $c=\overline{w_{n-1} w_n}$.
\end{itemize}
See Figure \ref{fig:P-n} for the case $n=5$. 

\begin{figure}[h]
    \centering
    \begin{tikzpicture}[scale=0.9]\footnotesize
\tikzset{every loop/.style={looseness=30}}
\node (1) at (1,2) [draw,circle,inner sep=1pt,fill=black,label=90:{$v_1$}] {};
\node (2) at (3,2)  [draw,circle,inner sep=1pt,fill=black,label=90:{$v_2$}] {};
\node (3) at (5,2) [draw,circle,inner sep=1pt,fill=black,label=90:{$v_3$}] {};
\node (4) at (7,2) [draw,circle,inner sep=1pt,fill=black,label=90:{$v_4$}] {};
\node (5) at (9,2) [draw,circle,inner sep=1pt,fill=black,label=90:{$v_5$}] {};
\node (6) at (1,0) [draw,circle,inner sep=1pt,fill=black,label=270:{$w_1$}] {};
\node (7) at (3,0) [draw,circle,inner sep=1pt,fill=black,label=270:{$w_2$}] {};
\node (8) at (5,0) [draw,circle,inner sep=1pt,fill=black,label=270:{$w_3$}] {};
\node (9) at (7,0) [draw,circle,inner sep=1pt,fill=black,label=270:{$w_4$}] {};
\node (10) at (9,0) [draw,circle,inner sep=1pt,fill=black,label=270:{$w_5$}] {};

\path
(1) edge [-, black] node[above=-0] {$a_1$} (2)
(2) edge [-, black] node[above=-0] {$a_2$} (3)
(3) edge [-, black] node[above=-0] {$a_3$} (4)
(4) edge [-, black] node[above=0] {$a_4$} (5)
(1) edge [-, black] node[right=-0] {$b_1$} (6)
(2) edge [-, black] node[right=0] {$b_2$} (7)
(3) edge [-, black] node[right=0] {$b_3$} (8)
(4) edge [-, black] node[right=0] {$b_4$} (9)
(5) edge [-, black] node[right=0] {$b_5$} (10)
(9) edge [-, black] node[above=0] {$c$} (10);

\end{tikzpicture}
    \caption{Graph $F_5$}
    \label{fig:P-n}
\end{figure}
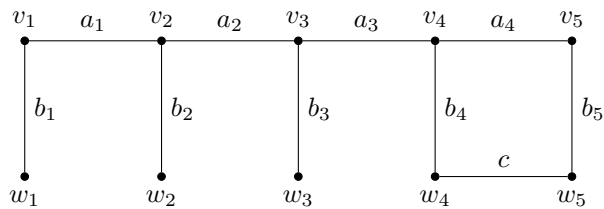 

The Lie bracket on $\n_{F_n}$ is given by
\[ [v_k,v_{k+1}]=a_k, \quad [v_k,w_k]=b_k, \quad [w_{n-1},w_n]=c.\]
Define an almost complex structure $J$ on $\n_{F_n}$ as follows:
\begin{gather*}
    Jv_1=w_n, \quad Jv_k=w_{k-1}, \quad k=2,\ldots, n, \\
    Jb_n=c, \quad Ja_k=b_k, \quad k=1, \ldots, n-1.
\end{gather*} 
Let us verify the integrability of $J$: first,
\begin{align*} 
N_J(v_1,v_2) & =[v_1,v_2]+J([w_n,v_2]+[v_1,w_1])-[w_n,w_1] \\
& = a_1+Jb_1 \\
&=0.
\end{align*}
For $k=3,\ldots, n-1$, it is clear that $N_J(v_1,v_k)=0$, since there are no edges connecting the vertices involved in the calculation. Next, we compute
\begin{align*} 
N_J(v_1,v_n) & =[v_1,v_n]+J([w_n,v_n]+[v_1,w_{n-1}])-[w_n,w_{n-1}] \\
& = -Jb_n+c \\
&=0.
\end{align*}
For $2\leq k\leq n-1$, we have
\begin{align*} 
N_J(v_k,v_{k+1}) & =[v_k,v_{k+1}]+J([w_{k-1},v_{k+1}]+[v_k,w_{k}])-[w_{k-1},w_{k}] \\
& = a_k+Jb_k \\
&=0.
\end{align*}
It is clear that $N_J(v_k,v_j)=0$ for $k+1\leq j\leq n-1$, since there are no edges connecting the vertices involved in the computation. It follows from \eqref{eq: simetrias} that $J$ is integrable.

Note that $w_n=Jv_1$ and $d(v_1,w_n)=n\geq 3$ with $\deg(v_1)=\deg(w_n)=2$, in accordance with Proposition \ref{prop: vertices impares lejos}.
\end{example}

\section{Construction of graphs with adapted complex structures}\label{sec:construction}

In this section, we prove the most important results of the article. First, we characterize the graphs that admit an abelian adapted complex structure, and call them \textit{basic}. Later, we show that any graph equipped with an adapted complex structure has a unique spanning subgraph that is basic. Conversely, we exhibit a method to construct graphs with an adapted complex structure beginning with a basic graph and adding certain pairs of edges. This provides an easy way to obtain 2-step nilpotent Lie algebras with complex structures.

\subsection{Abelian adapted complex structures and basic graphs}
Here, we will determine the graphs $\B$ that admit an adapted complex structure $J$ that is, in addition, abelian. Recall that this condition means that $[Jx,Jy]=[x,y]$ for all $x,y\in \n_\G$. We will show in the next result that this condition is very restrictive, since the only graphs that admit an abelian adapted complex structure are those obtained as the union of some copies of $\A_1$, $\A_2$ or $\A_3$ from Example \ref{ex: first examples}.

\begin{proposition}\label{prop: basic}
Let $\B=\B(V,E)$ be a graph with an adapted complex structure $J$. Then the following statements are equivalent:
\begin{enumerate}[{\rm (i)}]
    \item All the edges in $\B$ are distinguished,
    \item $\B$ is the union of some copies of $\A_1$, $\A_2$ or $\A_3$, and each copy is equipped with the adapted complex structure given in Example \ref{ex: first examples},
    \item The adapted complex structure $J$ on $\n_\B$ is abelian.
\end{enumerate}
\end{proposition}

\begin{proof}
If $\B$ is an empty graph with $2n$ vertices, then the equivalences hold trivially (since $\B$ is the union of $n$ copies of $\A_1$ and $\n_\B$ is abelian). Thus, from now on, we assume that $|E|\geq 1$.

$\ri \Rightarrow \rii$ Assume that all edges of $\B$ are distinguished. In particular, $\deg(v)\leq 1$ for any $v\in V$; moreover, the connected components of $\B$ are isolated vertices or segments. 
If $v$ is an isolated vertex of $\B$ then $Jv$ is, up to sign,  either another isolated vertex or an edge, due to Corollary \ref{cor: isolated}. In the former case, the subgraph formed only by the vertices $v$ and $Jv$ is isomorphic to $\A_1$, and it is invariant by $J$. On the other hand, in the latter case the $J$-invariant subgraph formed by $v$, $Jv$ and the endpoints of $Jv$ is isomorphic to $\A_2$. Now, we may decompose $E$ as a disjoint union $E=E_0\cup E_1$, where $E_0$ is set of edges that are mapped to (isolated) vertices by $J$ and $E_1$ is the set of edges that are mapped to another edge by $J$. Thus, if $a$ is an edge in $E_1$ then $Ja$ is another edge in $E_1$ and the $J$-invariant subgraph determined by $a, Ja$ and their endpoints is isomorphic to $\A_3$. This completes all possible cases, and thus $\B$ is the union of some copies of $\A_1$, $\A_2$ and $\A_3$, where each copy is equipped with the adapted complex structure given in Example \ref{ex: first examples}.

$\rii \Rightarrow \riii$ Assume that $\B$ is the union of $p_1$ copies of $\A_1$, $p_2$ copies of $\A_2$ and $p_3$ copies of $\A_3$ (with $p_1,p_2,p_3\in \N\cup \{0\})$. Then we have that 
\[ \n_\B=\underbrace{\R^2\times  \cdots \times \R^2}_{p_1} \times  \underbrace{(\h_3\times \R)\times \cdots \times (\h_3\times \R)}_{p_2} \times   \underbrace{(\h_3\times \h_3)\times \cdots \times (\h_3\times \h_3)}_{p_3}, \]
where each factor $\R^2$, $\h_3\times\R$ and $\h_3\times \h_3$ is $J$-invariant. Therefore, the adapted complex structure $J$ on $\n_\B$ is abelian, since it is on each factor.

$\riii \Rightarrow \ri$ Assume now that the adapted complex structure $J$ on $\n_\B$ is abelian. The condition for $J$ being abelian means $[Jx,Jy]=[x,y]$ for all $x,y\in \n_\B$, and this is equivalent to 
\[ [Jx,y]+[x,Jy]=0 \quad \text{ for all } x,y\in \n_\B. \]
If $x$ and $y$ are vertices in $\B$ connected by an edge then it follows from Corollary \ref{cor: distinguished} that $\overline{xy}$ is a distinguished edge. Hence, any edge is distinguished.
\end{proof}

The graphs appearing in Proposition \ref{prop: basic} will be very important throughout this article, so we give them a special name.

\begin{definition}\label{def: basic}
A graph $\B$ is called \textit{basic} if it is the union of several copies of the graphs $\A_1$, $\A_2$ and $\A_3$ from Example \ref{ex: first examples}.
\end{definition}

If $\B$ is a basic graph, we will use the notation $\B=(p_1,p_2,p_3)$, with $p_i\in \N_0$, to indicate that $\B$ is the union of $p_1$ copies of $\A_1$, $p_2$ copies of $\A_2$ and $p_3$ copies of $\A_3$. For example, $\A_1=(1,0,0)$, $\A_2=(0,1,0)$ and $\A_3=(0,0,1)$. Note that $2p_1+3p_2+4p_3$ is the number of vertices in $\B$. 

According to Proposition \ref{prop: basic}, each basic graph admits an abelian adapted complex structure, and they are the only graphs that admit them. However, we show next that a basic graph may admit different adapted complex structures.

\begin{remark}
The same basic graph may admit non-equivalent adapted complex structures. Indeed, consider the graph $\B$ given in Figure \ref{fig: 2 basic}.
\begin{figure}[h]
    \centering
    \begin{tikzpicture}\footnotesize
\tikzset{every loop/.style={looseness=30}}
\node (1) at (-1,0) [draw,circle,inner sep=1pt,fill=black!100,label=90:{$v_1$}] {};
\node (2) at (1,0)  [draw,circle,inner sep=1pt,fill=black!100,label=90:{$v_2$}] {};
\node (3) at (3,0) [draw,circle,inner sep=1pt,fill=black!100,label=90:{$v_3$}] {};
\node (4) at (-1,-1) [draw,circle,inner sep=1pt,fill=black!100,label=270:{$v_4$}] {};
\node (5) at (1,-1)  [draw,circle,inner sep=1pt,fill=black!100,label=270:{$v_5$}] {};
\node (6) at (3,-1) [draw,circle,inner sep=1pt,fill=black!100,label=270:{$v_6$}] {};

\path
(1) edge [-, black] node[above=-0] {} (2)
(4) edge [-, black] node[above=-0] {} (5);

\end{tikzpicture}
   \caption{Graph $\A_2\cup \A_2$ or $\A_3\cup \A_1$}
   \label{fig: 2 basic}
\end{figure} 
This graph can be seen as $\A_2\cup \A_2$, that is $\B=(0,2,0)$, where one $\A_2$ is given by the vertices $v_1,v_2,v_3$ and the edge $e_{1,2}$, and the second $\A_2$ is given by the vertices $v_4,v_5,v_6$ and the edge $e_{4,5}$. The corresponding adapted complex structure $J_1$ on $\n_\B$ satisfies $J_1v_1=v_2$, $J_1v_4=v_5$, $J_1v_3=e_{1,2}$ and $J_1v_6=e_{4,5}$. On the other hand, $\B$ can also be seen as $\A_3\cup \A_1$, that is, $\B=(1,0,1)$, where $\A_3$ is given by the vertices $v_i$, $1\leq i\leq 4$, and the edges $e_{1,2}, e_{4,5}$, and $\A_1$ consists only of the vertices $v_3,v_6$. The corresponding adapted complex structure $J_2$ on $\n_\B$ satisfies $J_2 v_1=v_2$, $J_2v_4=v_5$, $J_2v_3=v_6$ and $J_2e_{1,2}=e_{4,5}$. Clearly, the complex structures $J_1$ and $J_2$ are not equivalent since the commutator ideal of $\n_\B$ (generated by the two edges) is not invariant by $J_1$ but is invariant by $J_2$.
\end{remark}

\subsection{An exhaustive method to generate all graphs with adapted complex structures}

In this section, we show that in any graph $\G$ equipped with an adapted complex structure $J$ there is a unique $J$-invariant spanning basic subgraph $\G_b$. After that, we show conversely that any graph admitting an adapted complex structure can be obtained from a basic graph by adding certain edges in a precise way. 

We begin this section with the following important result, establishing the existence of a basic subgraph associated to the adapted complex structure.

\begin{theorem}\label{thm: existe basico}
Let $J$ be an adapted complex structure on the graph $\G$. Then there exists a unique $J$-invariant spanning basic subgraph $\G_b$ of $\G$. 
\end{theorem}

\begin{proof}
Let $E'$ denote the set of non-distinguished edges in $(\G,J)$, and let $\h$ denote the central ideal of $\n_\G$ generated by $E'$. Note that, according to Proposition \ref{prop: endpoint}, $\h$ is $J$-invariant, hence the quotient Lie algebra $\n_\G/\h$ carries an induced complex structure $J'$.

On the other hand, let $\G_b$ be the subgraph  of $\G$ with the same set of vertices $V$ and set of edges $E-E'$. That is, it is the spanning subgraph of $\G$ whose edges are precisely all the distinguished edges of $\G$. Note that $V\cup (E-E')$ is $J$-invariant, so that the Lie algebra $\n_{\G_b}$ admits an adapted almost complex structure $J_b$, induced by $J$. Now, let us define $f\colon \n_\G \to \n_{\G_b}$ in the following way:
\begin{itemize}
\item if $v\in V$ then $f(v)=v$,
\item if $e\in E-E'$ then $f(e)=e$,
\item if $e\in E'$ then $f(e)=0$. 
\end{itemize}
Then $f$ is a Lie epimorphism with $\ker f=\h$, thus it induces a Lie isomorphism $\overline{f}\colon \n_\G/\h \to \n_{\G_b}$. Moreover, $\overline{f}\circ J'=J_b\circ \overline{f}$. This implies that $J_b$ is integrable on $\n_{\G_b}$, hence $J_b$ is an adapted complex structure on $\G_b$. Since all the edges in $(\G_b,J_b)$ are distinguished, it follows from Proposition \ref{prop: basic} that $\G_b$ is basic. 

The uniqueness is clear.
\end{proof}

\begin{definition}
    The subgraph $\G_b$ from Theorem \ref{thm: existe basico} will be called the basic subgraph associated with $(\G,J)$. 
\end{definition}

\begin{example}\label{ex: A1}
In the graphs with adapted complex structures given in Proposition \ref{prop: bipartite} and Example \ref{ex: peine} there are no distinguished edges. Therefore, the associated basic subgraphs $\G_b$ are of the form $(p_1,0,0)$, that is, an empty graph obtained as a union of copies of $\A_1$.
\end{example}

\begin{example}
In this example we determine the basic subgraphs associated with the graphs in Example \ref{ex:G1G2}. 

\noindent (i) For $(\G_1,J_1)$, the associated basic subgraph is $\G_b=(0,0,1)$, shown in the figure below:
\begin{center}
\begin{tikzpicture}[scale=1]\footnotesize
    \node (1) at (0,0) [draw,circle,inner sep=1pt,fill=black!100,label=270:{$v_4$}] {};
    \node (2) at (2,0) [draw,circle,inner sep=1pt,fill=black!100,label=270:{$v_3$}] {};
    \node (3) at (2,1) [draw,circle,inner sep=1pt,fill=black!100,label=90:{$v_2$}] {};
    \node (4) at (0,1) [draw,circle,inner sep=1pt,fill=black!100,label=90:{$v_1$}] {};

    \path
    (1) edge  (2)
    (3) edge  (4);
\end{tikzpicture}
\end{center}

\noindent (ii) For $(\G_2,J_2)$, the associated basic subgraph is $\G_b=(1,1,0)$, shown in the figure below:
\begin{center}
\begin{tikzpicture}[scale=1]\footnotesize
    \node (1) at (0,0) [draw,circle,inner sep=1pt,fill=black!100,label=270:{$v_3$}] {};
    \node (2) at (2,0) [draw,circle,inner sep=1pt,fill=black!100,label=270:{$v_5$}] {};
    \node (3) at (2,1) [draw,circle,inner sep=1pt,fill=black!100,label=90:{$v_4$}] {};
    \node (4) at (0,1) [draw,circle,inner sep=1pt,fill=black!100,label=90:{$v_2$}] {};
    \node (5) at (-1,0.5) [draw,circle,inner sep=1pt,fill=black!120,label=180:{$v_1$}] {};
    \path
    (4) edge  (1);
\end{tikzpicture}
\end{center}
The subgraph $\A_1$ is the empty graph with vertices $v_4,v_5$, while $\A_2$ is the subgraph that has vertices $v_1,v_2,v_3$ and the only edge $e_{2,3}$.
\end{example}

\medskip

Conversely, we show next that any graph admitting an adapted complex structure can be obtained from a basic graph by adding certain edges in a precise way. 

We will need the following definition.

\begin{definition}
Let $\G$ be a graph equipped with an adapted complex structure $J$. A \textit{complex wedge} in $\G$ is a wedge $(u;v,w)$ of $\G$ such that $Jv=w$ and $J(\overline{uv})=\pm \, \overline{uw}$, in such a way that $[u,v]+J[u,w]=0$. 
\end{definition}

Graphically, a complex wedge is
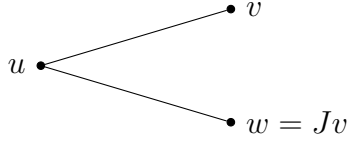
\[
\begin{tikzpicture}
\tikzset{every loop/.style={looseness=30}}
\node (1) at (-3,0) [draw,circle,inner sep=1pt,fill=black,label=180:{$u$}] {};
\node (2) at (-0.5,0.75) [draw,circle,inner sep=1pt,fill=black,label=0:{$v$}] {};
\node (3) at (-0.5,-0.75) [draw,circle,inner sep=1pt,fill=black,label=0:{$w=Jv$}] {};

\path
(1) edge [-, black] (2)
(1) edge [-, black] (3);
\end{tikzpicture}
\] 
where $J$ sends one of these edges to the other. We point out that the vertices $v$ and $Jv$ may or may not be adjacent. 

\begin{remark}
Let $(\G,J)$ be a graph with an adapted complex structure. Assume that the vertices in $\G$ are labeled $v_1,\ldots,v_n$ and that the Lie brackets of $\n_\G$ are given by $[v_i,v_j]=e_{i,j}$ if $i<j$, whenever $v_i$ and $v_j$ are connected by an edge. If $(v_k;v_i,v_j)$ is a wedge in $\G$ with $Jv_i=v_j$ then it is easy to verify that it is a complex wedge if and only if:
\begin{equation*}%\label{eq: Jtilde} 
J e_{i,k}= \begin{cases}
    \phantom{-} e_{j,k}, \quad \text{if } k<i,j \text{ or } k>i,j; \\ 
    -e_{j,k}, \quad \text{if } i<k<j \text{ or } j<k<i.
\end{cases} 
\end{equation*}
\end{remark}

\begin{example}
Consider the 4-cycle $C_4$ as in the picture, with the adapted complex structure $J$ given by $Jv_1=v_3$, $Jv_2=v_4$, $J e_{1,2}=e_{1,4}$ and $Je_{2,3}=-e_{3,4}$.
\[
\begin{tikzpicture}\footnotesize
\tikzset{every loop/.style={looseness=30}}
\node (1) at (0,0) [draw,circle,inner sep=1pt,fill=black!100,label=180:{$v_4$}] {};
\node (2) at (1.5,0)  [draw,circle,inner sep=1pt,fill=black!100,label=0:{$v_3$}] {};
\node (3) at (1.5,1.5) [draw,circle,inner sep=1pt,fill=black!100,label=0:{$v_2$}] {};
\node (4) at (0,1.5) [draw,circle,inner sep=1pt,fill=black!100,label=180:{$v_1$}] {};

\path
(1) edge [-, black] (2)
(1) edge [-, black] (4)
(2) edge [-, black] (3)
(3) edge [-, black] (4);
\end{tikzpicture} 
\] 
Then there are two complex wedges in $C_4$, namely $(v_1;v_2,v_4)$ and $(v_3;v_2,v_4)$. On the other hand, if we had chosen the adapted complex structure $J'$ given by $J'v_1=v_3$, $J'v_2=v_4$, $J'e_{1,2}=-e_{2,3}$ and $J'e_{1,4}=e_{3,4}$, the complex wedges would be $(v_2;v_1,v_3)$ and $(v_4;v_1,v_3)$.
\end{example}

With this new notion at hand, we may reinterpret Proposition \ref{prop: endpoint} as follows:

\begin{proposition}\label{prop:wedge}
Let $J$ be an adapted complex structure on $\G$. Then any non-distinguished edge is in a unique complex wedge in $\G$.
\end{proposition}

Furthermore, we observe that the proof of Theorem \ref{thm: existe basico} could have been carried out in stages, removing one complex wedge at a time. Using this idea, we will show how to build graphs with adapted complex structures beginning with a basic graph and adding several complex wedges.

\medskip

Let $\B$ be a basic graph equipped with an adapted complex structure $J$ (which is necessarily abelian according to Proposition \ref{prop: basic}). Let us add pairs of edges to $\B$ in the following way: 
\begin{itemize}
\renewcommand{\labelitemi}{$\diamond$}
    \item choose a vertex $u$ and a pair of vertices $v,w$ with $w=Jv$ (and $u\neq v,\, u\neq w$),
    \item add the edges $\overline{uv}$ and $\overline{uw}$,
    \item repeat this process several times; note that the vertices $u,v,w$ can be used again, as long as they are used to add edges different from $\overline{uv}$ and $\overline{uw}$. 
\end{itemize}

Let us denote by $\widetilde{\B}$ the graph obtained after adding several of these wedges; we will call $\widetilde{\B}$ an \textit{expansion} of the basic graph $\B$. Next, we extend the adapted complex structure $J$ on $\B$ to an almost complex structure $\widetilde{J}$ on $\widetilde{\B}$. We only need to specify the action of $\widetilde{J}$ on the new edges, and we do this in the following way:
\begin{itemize}
\renewcommand{\labelitemi}{$\diamond$}
    \item label arbitrarily the vertices of $\B$ (or $\widetilde{\B}$) as $v_1,\ldots, v_n$,
    \item consider a pair of new edges $e_{i,k}=\overline{v_iv_k}$ and $e_{j,k}=\overline{v_jv_k}$ with $Jv_i=v_j$, and then set:
\begin{equation}\label{eq: Jtilde} 
\widetilde{J} e_{i,k}= \begin{cases}
    \phantom{-} e_{j,k}, \quad \text{if } k<i,j \text{ or } k>i,j; \\ 
    -e_{j,k}, \quad \text{if } i<k<j \text{ or } j<k<i,
\end{cases} \quad \text{and} \quad (\widetilde{J})^2=-\operatorname{Id}.
\end{equation}
\end{itemize}

\begin{theorem}\label{thm: expanded graph}
The almost complex structure $\widetilde{J}$ defined above is integrable; hence, it defines an adapted complex structure on the expansion $\widetilde{\B}$ of $\B$ which extends the adapted complex structure $J$ on $\B$.
\end{theorem}

The proof of this theorem is an easy consequence of the following result.

\begin{lemma}
Let $\G=\G(V,E)$ be a graph equipped with an adapted complex structure $J$, and let $\widetilde{\G}$ be the graph obtained from $\G$ after adding a complex wedge $(v_k;v_i,v_j)$ for some $v_k,v_i,v_j\in V$ with $Jv_i=v_j$. Then the almost complex structure $\widetilde{J}$ on $\widetilde{\G}$ defined as in \eqref{eq: Jtilde} is integrable. 
\end{lemma}

\begin{proof}
Let us denote $\n:=\n_\G$ and $\widetilde{\n}=\n_{\widetilde{\G}}$. For simplicity, denote by $x:=e_{i,k}=\overline{v_iv_k}$ and $y:=e_{j,k}=\overline{v_jv_k}$ the edges in the new complex wedge. Then $\widetilde{\n}=\n\oplus \a$ as vector spaces, where $\a$ is the central ideal generated by $x$ and $y$, and the Lie bracket on $\widetilde{\n}$ is the same as the one on $\n$, extended by $[v_i,v_k]=\pm x$, $[v_j,v_k]=\pm y$, the signs depending on the enumeration of the vertices. This implies that the Chevalley-Eilenberg differential $d$ for $\alt^*\n^*$ is the restriction of the Chevalley-Eilenberg differential $d$ for $\alt^*\widetilde{\n}^*$.

Moreover, we have that $\widetilde{J}v_i=v_j$ and $\widetilde{J}x=\pm y$, according to \eqref{eq: Jtilde}, so that the decomposition $\widetilde{\n}=\n\oplus \a$ is $J$-invariant. We will verify the integrability of $\widetilde{J}$ on $\widetilde{\n}$ by showing that \eqref{eq: diff} holds for $\widetilde{\n}$. We point out that if $\alpha$ is a $(1,0)$-form on $(\n^\C)^*$ with respect to $J$ then $\alpha$ is still a $(1,0)$-form on $(\widetilde{\n}^\C)^*$ with respect to $\widetilde{J}$. Therefore, the component $(d\alpha)^{0,2}=0$, since $J$ is integrable on $\n$. 

Hence, we only need to check that if $\gamma$ is a generator of $\a^{*1,0}$, then the $(d\gamma)^{0,2}=0$. Let $V\cup E\cup \{x,y\}$ be the basis of $\widetilde{\n}$ and let $V^*\cup E^* \cup \{x^*,y^*\}$ be the dual basis  of $\widetilde{\n}^*$, with $x^*,y^*$ corresponding to $x,y$. Then, we have to consider some cases:
\begin{enumerate}
\item[$\ri$] $k<i,j$: in this case we have $[v_k,v_i]=x,\, [v_k,v_j]=y$ and $Jx=y$, hence we may choose $\gamma=x^*+\sqrt{-1} y^*$ as the generator of $\a^{*1,0}$. Then
\begin{align*}
d\gamma & = dx^*+\sqrt{-1}\, dy^* \\
        & = -v^k\wedge v^i- \sqrt{-1} v^k\wedge v^j \\
        & = -v^k \wedge (v^i+\sqrt{-1} v^j). 
\end{align*}
Since $v^i+\sqrt{-1} v^j$ is $(1,0)$, we have that $(d\gamma)^{0,2}=0$.

\item[$\rii$] $k>i,j$: the proof is similar to the one above, changing only $[v_k,v_i]=-x,\, [v_k,v_j]=-y$.

\item[$\riii$] $i<k<j$: in this case we have $[v_k,v_i]=-x,\, [v_k,v_j]=y$ and $Jx=-y$, hence we may choose $\gamma=y^*+\sqrt{-1} x^*$ as the generator of $\a^{*1,0}$. Then
\begin{align*}
d\gamma & = dy^*+\sqrt{-1}\, dx^* \\
        & = -v^k\wedge v^j+ \sqrt{-1} v^k\wedge v^i \\
        & = \sqrt{-1}\, v^k \wedge (v^i+\sqrt{-1} v^j).
\end{align*}
Since $v^i+\sqrt{-1} v^j$ is $(1,0)$, we have that $(d\gamma)^{0,2}=0$.

\item[$\riv$] $j<k<i$: the proof is similar to the one above.
\end{enumerate}
\end{proof}

\begin{example}
Consider the basic graph $\B=(1,1,1)$, where the subgraphs $\A_1$, $\A_2$ and $\A_3$ have vertex sets $\{v_1,v_2\}$, $\{v_3,v_4,v_5\}$ and $\{v_6,v_7,v_8,v_9\}$, respectively, as shown in the figure below.
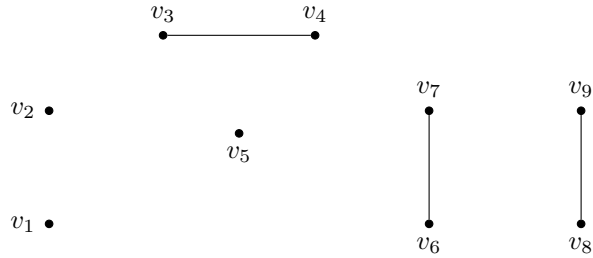
\begin{center}
\begin{tikzpicture}[scale=1]\footnotesize
    \node (1) at (-2.5,-0.5) [draw,circle,inner sep=1pt,fill=black!100,label=180:{$v_1$}] {};
    \node (2) at (-2.5,1) [draw,circle,inner sep=1pt,fill=black!100,label=180:{$v_2$}] {};
    \node (3) at (-1,2) [draw,circle,inner sep=1pt,fill=black!100,label=90:{$v_3$}] {};
    \node (4) at (1,2) [draw,circle,inner sep=1pt,fill=black!100,label=90:{$v_4$}] {};
     \node (5) at (0,0.7) [draw,circle,inner sep=1pt,fill=black!100,label=270:{$v_5$}] {};
      \node (6) at (2.5,-0.5) [draw,circle,inner sep=1pt,fill=black!100,label=270:{$v_6$}] {};
       \node (7) at (2.5,1) [draw,circle,inner sep=1pt,fill=black!100,label=90:{$v_7$}] {};
        \node (8) at (4.5,-0.5) [draw,circle,inner sep=1pt,fill=black!100,label=270:{$v_8$}] {};
     \node (9) at (4.5,1) [draw,circle,inner sep=1pt,fill=black!100,label=90:{$v_9$}] {};
     
    \path
    (3) edge (4)
    (6) edge (7)
    (8) edge (9);
 \end{tikzpicture}
 \end{center}
There is an (abelian) adapted complex structure $J$ on $\B$, obtained as in Example \ref{ex: first examples}. By adding various complex wedges (in dashed lines in the figure below), we obtain the following connected expansion $\widetilde{\B}$ of $\B$, which has an adapted complex structure $\widetilde{J}$.
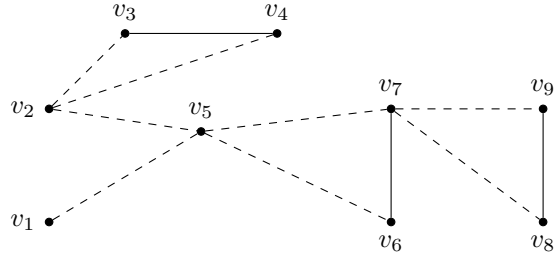
\begin{center}
 \begin{tikzpicture}[scale=1]\footnotesize
    \node (1) at (-2,-0.5) [draw,circle,inner sep=1pt,fill=black!100,label=180:{$v_1$}] {};
    \node (2) at (-2,1) [draw,circle,inner sep=1pt,fill=black!100,label=180:{$v_2$}] {};
    \node (3) at (-1,2) [draw,circle,inner sep=1pt,fill=black!100,label=90:{$v_3$}] {};
    \node (4) at (1,2) [draw,circle,inner sep=1pt,fill=black!100,label=90:{$v_4$}] {};
     \node (5) at (0,0.7) [draw,circle,inner sep=1pt,fill=black!100,label=90:{$v_5$}] {};
      \node (6) at (2.5,-0.5) [draw,circle,inner sep=1pt,fill=black!100,label=270:{$v_6$}] {};
       \node (7) at (2.5,1) [draw,circle,inner sep=1pt,fill=black!100,label=90:{$v_7$}] {};
        \node (8) at (4.5,-0.5) [draw,circle,inner sep=1pt,fill=black!100,label=270:{$v_8$}] {};
     \node (9) at (4.5,1) [draw,circle,inner sep=1pt,fill=black!100,label=90:{$v_9$}] {};

    \path
    (3) edge  (4)
    (6) edge  (7)
    (8) edge  (9)
    (2) edge [dashed] (3)
    (2) edge [dashed] (4)
    (1) edge [dashed] (5)
    (2) edge [dashed] (5)
    (5) edge [dashed] (6)
    (5) edge [dashed] (7)
    (7) edge [dashed] (8)
    (7) edge [dashed] (9);  
 \end{tikzpicture}
 \end{center}
According to \eqref{eq: Jtilde}, the complex structure $\widetilde{J}$ is given by:
\begin{multicols}{3}
\noindent $\widetilde{J}v_1=v_2$,\\
$\widetilde{J}v_3=v_4$,\\
$\widetilde{J}v_5=e_{3,4}$,\\
$\widetilde{J}v_6=v_7$,\\
$\widetilde{J}v_8=v_9$,\\
$\widetilde{J}e_{6,7}=e_{8,9}$,\\
$\widetilde{J}e_{1,5}=e_{2,5}$,\\
$\widetilde{J}e_{2,3}=e_{2,4}$,\\
$\widetilde{J}e_{5,6}=e_{5,7}$,\\
$\widetilde{J}e_{7,8}=e_{7,9}$.\\
 \end{multicols}
\end{example}

\medskip

It follows from Theorems \ref{thm: existe basico} and \ref{thm: expanded graph} (together with Proposition \ref{prop:wedge}) that:
\begin{enumerate}[(a)]
    \item given a graph $\G$ with an adapted complex structure $J$, if $\G_b$ is the basic subgraph associated with $(\G,J)$ and $J_b$ is the restriction of $J$ to $\G_b$, then an appropriate expansion of $(\G_b,J_b)$ recovers $(\G,J)$;
    \item given a basic graph $\B$ with an adapted complex structure $J$, if $(\widetilde{\B}, \widetilde{J})$ is any expansion of $(\B,J)$ then the basic subgraph associated with $(\widetilde{\B}, \widetilde{J})$ is $\B$.
\end{enumerate}

\medskip

As a consequence, we may state the main result of this section:

\begin{theorem}\label{thm:characterization}
Any graph equipped with an adapted complex structure is an expansion of a basic graph.
\end{theorem}

Thus, we have obtained a simple method with which we can produce \textit{all} graphs that admit an adapted complex structure. In this way, we are able to easily obtain many examples of 2-step nilpotent Lie algebras endowed with complex structures. 

However, it is not clear yet how to determine whether a given graph admits an adapted complex structure or not. In certain particular cases, as in Example \ref{ex:moser} below, this is an easy task, but in more general and sophisticated graphs we still cannot tell in a simple way if such a complex structure exists.

\begin{example}\label{ex:moser}
The graph $M$ in Figure \ref{fig:moser} below is known as the \textit{Moser spindle}, and has had several applications in graph theory. Its graphic depiction allows us to easily see that it admits an adapted complex structure. Indeed, we can choose the vertices $v_1,v_6,v_7$ and the edge $e_{6,7}$ as the vertices and edges of a subgraph $\A_2$, and the vertices $v_2,v_3,v_4,v_5$ and the edges $e_{2,3}, e_{4,5}$ as the vertices and edges of a subgraph $\A_3$. We can consider the other edges in $M$ as added in the expansion process. We define a complex structure on these subgraphs as in Example \ref{ex: first examples}, and then expand it to the whole of $M$ using \eqref{eq: Jtilde}. Thus, an adapted complex structure on $M$ is given by
\begin{gather*} 
Jv_1=e_{6,7}, \quad Jv_6=v_7, \quad Jv_2=v_3, \quad Jv_4=v_5, \quad Je_{2,3}=e_{4,5},   \\
Je_{1,2}=e_{1,3}, \quad Je_{6,2}=e_{6,3}, \quad Je_{1,4}=e_{1,5}, \quad Je_{7,4}=e_{7,5}.
\end{gather*}

\begin{figure}[h]
\begin{tikzpicture}[scale=2]\footnotesize

% --- Coordenadas de los vértices ---
\coordinate (A) at (0,0);
\coordinate (B) at (-0.92,-0.38);
\coordinate (C) at (-0.38,-0.92);
\coordinate (D) at (0.38,-0.92);
\coordinate (E) at (0.92,-0.38); 
\coordinate (F) at (-1.3,-1.3);
\coordinate (G) at (1.3,-1.3);

% --- Relleno de pseudotriángulos ---
%\fill[blue!10] (A)--(B)--(C)--cycle;
%\fill[blue!10] (B)--(C)--(D)--cycle;
%\fill[blue!10] (C)--(D)--(E)--cycle;
%\fill[blue!10] (C)--(E)--(F)--cycle;
%\fill[blue!10] (C)--(F)--(A)--cycle;
%\fill[blue!10] (C)--(D)--(G)--cycle;
%\fill[blue!10] (C)--(G)--(B)--cycle;

% --- Aristas ---
\draw (A)--(B)--(C)--cycle;
\draw (B)--(C)--(F)--cycle;
\draw (A)--(E)--(D)--cycle;
\draw (D)--(G)--(E)--cycle;
\draw (F)--(G);

% --- Vértices ---
\foreach \p in {A,B,C,D,E,F,G}
  \fill (\p) circle (1pt);

% --- Etiquetas opcionales ---
\node[above] at (A) {$v_1$};
\node[left] at (B) {$v_2$};
\node[below] at (C) {$v_3$};
\node[below] at (D) {$v_4$};
\node[right] at (E) {$v_5$};
\node[left] at (F) {$v_6$};
\node[right] at (G) {$v_7$};

\end{tikzpicture}
\caption{Moser spindle}
    \label{fig:moser}
\end{figure}
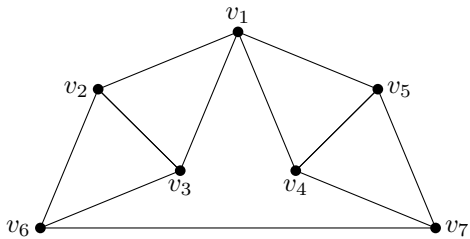
\end{example}

\begin{example}
In this example, we show that, beginning with a basic graph endowed with an adapted complex structure, we can expand it in two different ways so that the resulting graphs are isomorphic, but the expanded complex structures are not equivalent. 
Consider the basic graph $\B=(3,0,0)$, that is, the empty graph with vertices $V=\{v_1,\ldots,v_6\}$ with $Jv_{2i-1}=v_{2i}$, $1\leq i \leq 3$. For the first expansion, we add the complex wedges $(v_1;v_3,v_4)$ and $(v_5;v_3,v_4)$, obtaining in this way the graph $\G_1$ with adapted complex structure $J_1$. For the second expansion, we add the complex wedges $(v_1;v_3,v_4)$ and $(v_2;v_3,v_4)$, obtaining in this way the graph $\G_2$ with adapted complex structure $J_2$. See Figure \ref{fig: 2expansiones}.
\begin{figure}[h]
 \begin{tikzpicture}[scale=1.0]\footnotesize
    \node (1) at (0,1) [draw,circle,inner sep=1pt,fill=black!100,label=90:{$v_1$}] {};
    \node (2) at (0,0) [draw,circle,inner sep=1pt,fill=black!100,label=270:{$v_2$}] {};
    \node (3) at (1.5,1) [draw,circle,inner sep=1pt,fill=black!100,label=90:{$v_3$}] {};
    \node (4) at (1.5,0) [draw,circle,inner sep=1pt,fill=black!100,label=270:{$v_4$}] {};
    \node (5) at (3,1) [draw,circle,inner sep=1pt,fill=black!100,label=90:{$v_5$}] {};
    \node (6) at (3,0) [draw,circle,inner sep=1pt,fill=black!100,label=270:{$v_6$}] {};

    \path
    (1) edge (3)
    (1) edge (4)
    (3) edge (5)
    (5) edge (4);
 \end{tikzpicture}
 \hspace{3cm}
\begin{tikzpicture}[scale=1.0]\footnotesize
    \node (1) at (0,1) [draw,circle,inner sep=1pt,fill=black!100,label=90:{$v_1$}] {};
    \node (2) at (0,0) [draw,circle,inner sep=1pt,fill=black!100,label=270:{$v_2$}] {};
    \node (3) at (1.5,1) [draw,circle,inner sep=1pt,fill=black!100,label=90:{$v_3$}] {};
    \node (4) at (1.5,0) [draw,circle,inner sep=1pt,fill=black!100,label=270:{$v_4$}] {};
    \node (5) at (3,1) [draw,circle,inner sep=1pt,fill=black!100,label=90:{$v_5$}] {};
    \node (6) at (3,0) [draw,circle,inner sep=1pt,fill=black!100,label=270:{$v_6$}] {};

    \path
    (1) edge (3)
    (1) edge (4)
    (3) edge (2)
    (2) edge (4);
 \end{tikzpicture}
 \caption{Expanded graph $\G_1$ (left) and expanded graph $\G_2$ (right)}
    \label{fig: 2expansiones}
 \end{figure}
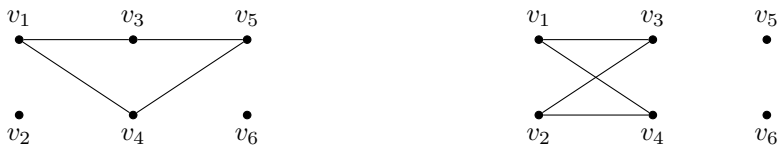
 Clearly, the graphs $\G_1$ and $\G_2$ are isomorphic, and therefore, the Lie algebras $\n_{\G_1}$ and $\n_{\G_2}$ are isomorphic. However, the complex structures $J_1$ and $J_2$ are not equivalent, since in $\n_{\G_1}$ the center is not $J_1$-invariant but in $\n_{\G_2}$ the center is $J_2$-invariant.
\end{example}

We end this section showing that, among a very well known family of graphs, only two of them admit adapted complex structures.

\begin{corollary}\label{cor:cycle}
The $n$-cycle $C_n$ admits an adapted complex structure if and only if $n=3$ or $n=4$. 
\end{corollary}

\begin{proof}
We label $v_1,\ldots,v_n$ the vertices of $C_n$ in a cyclic way and denote the edges of $C_n$ by $a_{i}=\overline{v_iv_{i+1}}$ ($i=1,\ldots,n-1$) and $a_n=\overline{v_1v_n}$. As usual, we set the Lie bracket on $\n_{C_n}$ as $[v_i,v_{i+1}]=a_i$, $i=1,\ldots, n-1$, and $[v_1,v_n]=a_n$.

If $n=3$, it can easily be verified that the almost complex structure $J$ on $\n_{C_3}$ given by $Jv_1=a_2, \, Jv_2=v_3, \, Ja_1=a_3$ is indeed integrable, so that $J$ is an adapted complex structure. 

If $n=4$, we define an almost complex structure $J$ on $\n_{C_4}$ by
\[ Jv_1=v_3, \, Jv_2=v_4, \, Ja_2=a_1, \, Ja_4=a_3.  \]
We easily see that $J$ is integrable, so it is an adapted complex structure on $\n_{C_4}$. 

We will show next that there are no adapted complex structures on $\n_{C_n}$ when $n\geq 5$. Assume by contradiction that $C_n$, with $n\geq 5$, admits an adapted complex structure, and let $G_b=(p_1,p_2,p_3)$ be the corresponding basic subgraph from Theorem \ref{thm: existe basico}. We first prove that $p_2=p_3=0$. Indeed, if $p_2\neq 0$ and $v_i, v_j$ were the non-isolated vertices in an $\A_2$ subgraph, then $v_i$ cannot be the center of a complex wedge (since in this case we would have that $\deg(v_i)>2$). Now, as $\deg(v_i)=2$, there is a complex wedge with center $v_k$ ($k\neq i,j)$ that has $v_i$ and $v_j$ as endpoints, but then $v_k,v_i,v_j$ are the vertices of a 3-cycle in $\G$, which is not possible. Thus $p_2=0$. In an analogous way, we prove that $p_3=0$.

Therefore, $\G_b=(p_1,0,0)$ and then $n=2p_1$ is even, and clearly $p_1\geq 2$. Let $(v_k;v_i,v_j)$ be a complex wedge in $\G$, with $v_i,v_j$ the vertices of an $\A_1$ subgraph. Then, $v_i$ and $v_j$ cannot be the center of any other complex wedge in $\G$, because if $v_i$ were the center of a complex wedge then $\deg(v_i)$ would be greater than 2, which is impossible. Now, since $\deg(v_i)=\deg(v_j)=2$, there must be another complex wedge $(v_l;v_i,v_j)$, but now $v_k,v_i,v_l,v_j$ are the vertices of a 4-cycle in $\G$, which is not possible. 
\end{proof}

\section{Further results on adapted complex structures}\label{sec:further}

In this section, we establish some consequences of the characterization of graphs with adapted complex structures in terms of basic graphs given earlier.

\begin{lemma}\label{lem:misc}
Let $\G$ be a graph equipped with an adapted complex structure $J$, and let $\G_b=(p_1,p_2,p_3)$ the associated basic subgraph.
\begin{enumerate}
    \item[$\ri$] If $v$ is a vertex of $\G$ that in $\G_b$ is the isolated vertex of a subgraph $\A_2$ then $\deg (v)$ is even.
    \item[$\rii$] $p_2$ and $|V|$ have the same parity. In particular, if $|V|$ is odd then $p_2\geq 1$.
\end{enumerate}
\end{lemma}

\begin{proof}
(i) If $v$ is still isolated in $\G$ then $\deg(v)=0$, which is even. 
If $v$ is not isolated in $\G$, then all edges that have $v$ as an endpoint come in pairs, since if $v$ is connected to the vertex $w$ then this edge is part of a complex wedge of the form $(v;w,Jw)$, that is, there is also an edge connecting $v$ with $Jw$. Thus, $\deg (v)$ is even.

(ii) Since $\G_b$ is a spanning subgraph of $\G$ then we have that
\[ |V|=2p_1+3p_2+4p_3, \]
and the result follows.
\end{proof}

As a consequence of Lemma \ref{lem:misc}(i), if all the vertices of a graph $\G$ equipped with an adapted complex structure are odd, then in the corresponding basic subgraph $\G_b$ there is no subgraph $\A_2$, that is, $p_2=0$. In the following example, we exhibit a graph with an adapted complex structure whose vertices are all odd.

\begin{example}\label{ex:cubic}
We show a graph $\G$ equipped with an adapted complex structure such that all vertices are odd; more precisely, all have degree equal to 3 (that is, $\G$ is a 3-regular graph).
Consider the basic graph $\B=(4,0,0)$, where each copy of $\A_1$ is given by vertices $v_{2i-1},v_{2i}$, with $Jv_{2i-1}=v_{2i}$, for $1\leq i\leq 4$. We consider the graph $\G=\widetilde{\B}$ obtained by expanding $\B$ with the following complex wedges:
\[ (v_6;v_1,v_2), \, (v_8;v_3,v_4), \, (v_2;v_3,v_4), \, (v_5;v_3,v_4), \, (v_7;v_5,v_6), \, (v_1;v_7,v_8).  \]
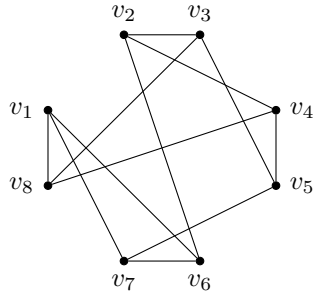
\begin{center}
 \begin{tikzpicture}\footnotesize
    \node (1) at (-1,1) [draw,circle,inner sep=1pt,fill=black!100,label=90:{$v_2$}] {};
    \node (2) at (0,1) [draw,circle,inner sep=1pt,fill=black!100,label=90:{$v_3$}] {};
    \node (3) at (1,0) [draw,circle,inner sep=1pt,fill=black!100,label=0:{$v_4$}] {};
    \node (4) at (1,-1) [draw,circle,inner sep=1pt,fill=black!100,label=0:{$v_5$}] {};
    \node (5) at (0,-2) [draw,circle,inner sep=1pt,fill=black!100,label=270:{$v_6$}] {};
    \node (6) at (-1,-2) [draw,circle,inner sep=1pt,fill=black!100,label=270:{$v_7$}] {};
    \node (7) at (-2,-1) [draw,circle,inner sep=1pt,fill=black!100,label=180:{$v_8$}] {};
    \node (8) at (-2,0) [draw,circle,inner sep=1pt,fill=black!100,label=180:{$v_1$}] {};
    \path
    (1) edge [-, black] (2)
    (1) edge [-, black] (3)
    (1) edge [-, black] (5)
    (2) edge [-, black] (7)
    (2) edge [-, black] (4)
    (3) edge [-, black] (4)
    (3) edge [-, black] (7)
    (4) edge [-, black] (6)
    (5) edge [-, black] (6)
    (5) edge [-, black] (8)
    (8) edge [-, black] (7)
    (8) edge [-, black] (6);
    \end{tikzpicture}   
\end{center}
The action of the complex structure $\widetilde{J}$ on the edges of $\G$ is easily established using \eqref{eq: Jtilde}.
\end{example}

\begin{remark}
The graph appearing in the previous example is a 3-regular (or cubic) graph with 8 vertices. Connected cubic graphs with up to 12 vertices were classified in \cite{BCCS}, and in particular there are 5 with 8 vertices. The graph in Example \ref{ex:cubic} is isomorphic to the following one (appearing in that classification), where the labeling of the vertices specifies the isomorphism between the two graphs: 
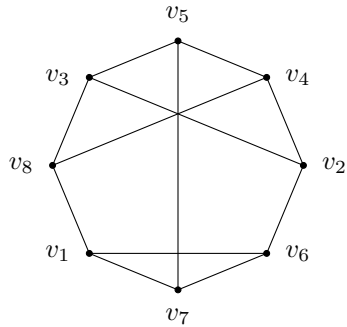
\begin{center}
\begin{tikzpicture}[scale=2.2]\footnotesize

% --- Vértices del octógono ---
\foreach \i in {0,...,7}{
  \coordinate (P\i) at (90+45*\i:3/4); % vértices en círculo de radio 1
}

% --- Octógono ---
\draw (P0)--(P1)--(P2)--(P3)--(P4)--(P5)--(P6)--(P7)--cycle;

% --- Vértices con circulitos ---
%\foreach \i in {0,...,7}{
 % \filldraw[white,draw=black] (P\i) circle (1.2pt); % circulito hueco
 % \node[font=\small,anchor=south] at (P\i) {$P_{\i}$};
  \filldraw[draw=black] (P1) circle (0.5pt); 
  \node[font=\small,label=180:{$v_3$}] at (P1) {};
  \filldraw[draw=black] (P2) circle (0.5pt); 
  \node[font=\small,label=180:{$v_8$}] at (P2) {};
  \filldraw[draw=black] (P3) circle (0.5pt); 
  \node[font=\small,label=180:{$v_1$}] at (P3) {};
  \filldraw[draw=black] (P4) circle (0.5pt); 
  \node[font=\small,label=270:{$v_7$}] at (P4) {};
  \filldraw[draw=black] (P5) circle (0.5pt); 
  \node[font=\small,label=0:{$v_6$}] at (P5) {};
  \filldraw[draw=black] (P6) circle (0.5pt); 
  \node[font=\small,label=0:{$v_2$}] at (P6) {};
  \filldraw[draw=black] (P7) circle (0.5pt); 
  \node[font=\small,label=0:{$v_4$}] at (P7) {};
  \filldraw[draw=black] (P0) circle (0.5pt); 
  \node[font=\small,label=90:{$v_5$}] at (P0) {};

% --- Ejemplo de diagonales ---
\draw[black] (P0)--(P4);
\draw[black] (P3)--(P5);
\draw[black] (P2)--(P7);
\draw[black] (P1)--(P6);

\end{tikzpicture}
\end{center}
\end{remark}

As a consequence of Lemma \ref{lem:misc}, we show that the graphs in a well-known family do not admit adapted complex structures.  

\begin{example}
We exhibit an infinite family of connected graphs that do not admit adapted complex structures. For $n\geq 2$, let $W_{2n}$ denote the wheel with $2n$ vertices, that is, $V=\{v_0,\ldots, v_{2n-1}\}$ where $v_1,\ldots, v_{2n-1}$ are the vertices of a $(2n-1)$-cycle $C_{2n-1}$, and $v_0$ is adjacent to $v_i$ for $1\leq i\leq 2n-1$ (see Figure \ref{fig:W8} for $n=4$). The edges in the cycle are $e_{1,2}, e_{2,3},\ldots, e_{2n-1,1}$, and the other edges are $e_{0,i}$ for $1\leq i\leq 2n-1$. 
Therefore, $|V|=2n$ and $|E|=4n-2$, so that $\dim \n_{W_{2n}}=6n-2$. Note that $\deg(v_0)=2n-1$ and $\deg(v_i)=3$ for $1\leq i\leq 2n-1$. Thus, all vertices are odd. 

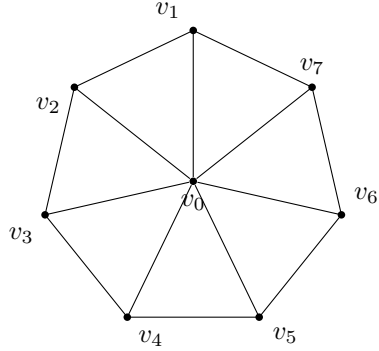
\begin{figure}[h]
\begin{tikzpicture}[scale=2]\footnotesize

% Definimos el número de lados
\def\n{7}
\def\r{1} % radio del circuncírculo

% Centro
\node[fill=black,circle,inner sep=1pt,label=below:$v_0$] (v0) at (0,0) {};

% Colocamos los vértices en un círculo
\foreach \i in {1,...,\n} {
  \coordinate (v\i) at ({\r*cos(90+360/\n*(\i-1))}, {\r*sin(90+360/\n*(\i-1))});
  % Dibujamos el vértice
  \node[fill=black,circle,inner sep=1pt,label={90+360/\n*\i:$v_{\i}$}] at (v\i) {};
  % Dibujamos el lado del polígono
  %\draw (v\i) -- (v\i+1);
  % Dibujamos el segmento desde el centro
  \draw (v0) -- (v\i);
}
\draw (v1) -- (v2)--(v3)--(v4)--(v5)--(v6)--(v7)--cycle;
\end{tikzpicture}
\caption{The graph $W_8$}\label{fig:W8}
\end{figure}

For $n\geq 3$, there is no adapted complex structure on $W_{2n}$. Indeed, assume by contradiction that $W_{2n}$ admits an adapted complex structure $J$, and let $\G_b=(p_1,p_2,p_3)$ be the associated basic graph. It follows from Lemma \ref{lem:misc}(i) that $p_2=0$. Hence, we have $2p_1+4p_3=|V|=2n$, that is, $p_1+2p_3=n$. Taking into account Corollary \ref{cor:odd}, we see that $Jv_0$ is another vertex of $W_{2n}$, and we may assume $Jv_0=\pm v_1$. In particular, the edge $e_{0,1}$ is a distinguished edge of $W_{2n}$; therefore, $p_3\geq 1$ and there is at least one more distinguished edge in $W_{2n}$. This other distinguished edge is different from $e_{1,2}$, $e_{1,2n-1}$ or $e_{0,i}$ with $i\neq 1$, since they are adjacent to the distinguished edge $e_{0,1}$. Suppose then that $e_{j,j+1}$ is another distinguished edge with $Je_{0,1}=\pm e_{j,j+1}$, with $2\leq j\leq 2n-2$ (these edges form an $\A_3$ subgraph). Then $(v_0;v_j,v_{j+1})$ is a complex wedge in $W_{2n}$. The edge $e_{j-1,j}$ is not distinguished (because it is adjacent to the distinguished edge $e_{j,j+1}$), and therefore it is in a complex wedge with center $v_{j-1}$ (it cannot have center $v_j$ since the edge $e_{0,j}$ is already part of a complex wedge and $\deg(v_j)=3$). Then, the other edge in this complex wedge has to be $e_{j-1,j+1}$. Since $n\geq 3$, we would have $\deg(v_{j+1})\geq 4$, which is absurd.  Hence, there is no adapted complex structure on $W_{2n}$.

On the other hand, it is easily verified that $W_4$ does admit an adapted complex structure (it is isomorphic to the complete graph $K_4$, see Proposition \ref{prop:complete} below).
\end{example}

\medskip

\subsection{Nilpotency of adapted complex structures}
Let $\g$ be a nilpotent Lie algebra equipped with a complex structure $J$. In \cite{CFGU} a natural ascending series related to $J$ was introduced. We recall this definition.

\begin{definition}
The $J$-compatible ascending series $\{\a_k \mid k \ge 0\}$ of $\g$ is defined inductively by
\[%begin{equation} \label{eqn:ascending}
 \a_0 = 0,\  \a_k = \{x\in \g \mid [x,\g] \subset \a_{k-1} \ \  \text{and} \ \ [Jx,\g] \subset
 \a_{k-1}\}, \ k\ge 1.  
\]%end{equation}
The complex structure $J$ is said to be \emph{nilpotent} if there exists $k \in \N$ such that $\a_k=\g$. The smallest such $k$ is called the \emph{nilpotency step of $J$}. 
\end{definition}

\par
If $\g$ is $2$-step nilpotent, then any complex structure $J$ is nilpotent \cite[Proposition 3.3]{Rol}, and the nilpotency step of $J$ is either $2$ or $3$ \cite[Theorem 1.3]{GZZ}. More precisely, this nilpotency step can be easily characterized according to the next result from \cite{B}.

\begin{lemma}[{\cite[Lemma 3.3]{B}}] \label{lem:nilpotencystep}
Let $\g$ be a $2$-step nilpotent Lie algebra endowed with a complex structure $J$. If $\g':=[\g,\g]$ denotes the commutator ideal of $\g$ (which is contained in the center $\z$ of $\g$), then
\begin{enumerate} 
    \item [$\ri$] $J$ is $2$-step nilpotent if and only if $J \g' \subset \z$,
    \item [$\rii$] $J$ is $3$-step nilpotent if and only if there exists $x \in \g'$ such that $Jx \notin \z$. 
%    \item [$(3)$]$J\g^1$ is an abelian ideal of $\g$. 
\end{enumerate}
\end{lemma}
Note that any abelian complex structure $J$ on a 2-step nilpotent Lie algebra is 2-step nilpotent (since the center is $J$-invariant).

We now consider our case of interest, that is, an adapted complex structure $J$ on a 2-step nilpotent Lie algebra associated with a graph $\G$. As a consequence of Lemma \ref{lem:nilpotencystep}, we can characterize the step of nilpotency of any adapted complex structure, in terms of the corresponding basic subgraph $\G_b$ from Theorem \ref{thm: existe basico}.

\begin{proposition}\label{prop:3-step}
Let $J$ be an adapted complex structure on the graph $\G$, and let $\G_b$ be the associated basic subgraph, with $\G_b=(p_1,p_2,p_3)$. Then the complex structure $J$ on $\n_\G$ is 3-step nilpotent if and only 
\begin{enumerate}
    \item[$\ri$] $p_2\geq 1$, and
    \item[$\rii$] the isolated vertex of one of the copies of $\A_2$ is not isolated when considered as a vertex of $\G$.
\end{enumerate}
\end{proposition}

\begin{proof}
Assume first that $J$ is 3-step nilpotent. According to Lemma \ref{lem:nilpotencystep} there exists an edge $a$ such that $Ja$ is not in the center of $\n_\G$, that is, $v:=Ja$ is, up to sign, a non-isolated vertex of $\G$. It follows from Lemma \ref{lem:central}  that $a$ is a distinguished edge of $\G$, and therefore it is also an edge in $\G_b$. Then the $J$-invariant subgraph consisting of the edge $a$, the endpoints of $a$ and the vertex $v$ is a graph isomorphic to $\A_2$. Thus, both $\ri$ and $\rii$ hold.

The converse follows readily, since the image by $J$ of the distinguished edge $a$ in the given copy of $\A_2$ is, up to sign, a non-isolated vertex of $\G$. Hence, $a$ is in the commutator ideal of $\n_\G$ but $Ja$ is not in the center of $\n_\G$, so $J$ is $3$-step nilpotent.
\end{proof}

\begin{corollary}
    With the same notation as in Proposition \ref{prop:3-step}, if $p_2\geq 1$ and $\G$ is connected then the adapted complex structure $J$ on $\G$ is $3$-step nilpotent.
\end{corollary}

\begin{example}
Consider the graph $\A_2$ from Example \ref{ex: first examples}, and expand it by adding the edges $e_{1,3}:=\overline{v_1v_3}$ and $e_{2,3}:=\overline{v_2v_3}$. This expansion is clearly the $3$-cycle graph $C_3$. Extending the original adapted complex structure $J$ on $\A_2$ by setting $Je_{1,3}=e_{2,3}$ we obtain an adapted complex structure on $C_3$. It follows from the previous corollary that this complex structure on $\n_{C_3}$ is $3$-step nilpotent.
\end{example}

\subsection{Complex structures on complete graphs}
The free 2-step nilpotent Lie algebra $\mathfrak{f}_m$ on $m$ generators is defined as $\mathfrak{f}_m=\R^m\oplus \alt^2 \R^m$, where the Lie bracket is defined as follows: $\alt^2 \R^m$ is the center of $\mathfrak{f}_m$ and for $x,y\in \R^m$, 
\[ [x,y]=x\wedge y \in \alt^2 \R^m. \]
The existence of complex structures on free 2-step nilpotent Lie algebras was settled in \cite[Proposition 4.9]{B}; indeed, it was proved that:
\begin{itemize}
    \item[$\diamond$] $\mathfrak{f}_m$ admits a complex structure when $m\equiv 0$ or $3$ (mod 4), 
    \item[$\diamond$] $\mathfrak{f}_m\times \R$ admits a complex structure when $m\equiv 1$ or $2$ (mod 4).
\end{itemize} 
We will obtain another proof of this result, exploiting the fact that the Lie algebra $\mathfrak{f}_m$ is isomorphic to the Lie algebra $\n_{K_m}$ associated with the complete graph $K_m$. 

\begin{proposition}\label{prop:complete}
Let $K_m$ denote the complete graph with $m$ vertices. Then, for any $n\in\N$, the graphs $K_{4n}$, $K_{4n+3}$, $(K_{4n+1})_*$ and $(K_{4n+2})_*$ admit adapted complex structures.
\end{proposition}

\begin{proof}
Let $n\in\N$ be fixed.\par
For $K_{4n}$, consider the basic graph $\B=(0,0,n)$. Then it is easily seen that expanding $\B$ by adding \textit{all} possible complex wedges, the expanded graph $\widetilde{\B}$ is precisely $K_{4n}$, so it admits an adapted complex structure, due to Theorem \ref{thm: expanded graph}.\par 
In a similar way, for $K_{4n+3}$ we begin with the basic graph $\B=(0,1,n)$. Then, again expanding $\B$ by adding \textit{all} possible complex wedges, the expanded graph $\widetilde{\B}$ is precisely $K_{4n+3}$, so it admits an adapted complex structure.\par 
For $K_{4n+1}$, we consider the basic graph $\B=(1,0,n)$. If the vertices of $\A_1$ are $v_0$ and $v_1$, we expand $\B$ by adding all possible pairs of edges between $v_1$ and the vertices in $n\A_3$, and then all possible pairs of edges between the vertices in $n\A_3$. Then the expanded graph $\widetilde{\B}$ is $(K_{4n+1})_*=K_{4n+1}\cup \{v_0\}$, so it admits an adapted complex structure. \par
Finally, for $K_{4n+2}$ we begin with the basic graph $\B=(0,1,n)$. In $\A_2$, let $v_0$ denote the isolated vertex and $v_1$, $v_2$ the non-isolated vertices. As before, expand $\B$ by adding all possible pairs of edges between $v_1$, $v_2$ and the vertices in $n\A_3$, and then all possible pairs of edges between the vertices in $n\A_3$. Then the expanded graph $\widetilde{\B}$ is $(K_{4n+2})_*=K_{4n+2}\cup \{v_0\}$, so it admits an adapted complex structure.
\end{proof}

%\begin{remark}
%For the study of the moduli space of complex structures on free 2-step nilpotent Lie algebras, see \cite{F-C}.
%\end{remark}

Recall that the chromatic number of a graph $\G$ is the smallest number of colors needed to paint the vertices of $\G$ so that no two adjacent vertices share the same color. For instance, the chromatic number of a bipartite graph is 2. Since the chromatic number of $K_m$ is $m$, the following result is an immediate consequence of Proposition \ref{prop:complete}.

\begin{corollary}
The chromatic number of a graph admitting an adapted complex structure can be arbitrarily large.
\end{corollary}

\subsection{Girth}

The girth of a graph is the length of a shortest cycle contained in the graph. If the graph is acyclic, that is, a forest, then its girth is defined as infinite. 

With respect to the girth of a graph that admits an adapted complex structure, we have the following result.

\begin{proposition}\label{prop:girth}
Let $\G$ be a graph that admits an adapted complex structure $J$, and let $\G_b=(p_1,p_2,p_3)$ be the corresponding basic subgraph. Then:
\begin{enumerate}
    \item[$\ri$] if $p_1=0$ and $\G$ is not basic then the girth of $\G$ is 3;
    \item[$\rii$] for each $n\in \N$, $n\geq 3$, we can choose $\G$ in such a way that it is connected, $p_2=p_3=0$ and its girth is $n$. 
\end{enumerate}
\end{proposition}

\begin{proof}
$\ri$ Since there are no copies of $\A_1$ in $\G_b$, for any complex wedge $(u;v,Jv)$ in $\G$ the endpoints $v$ and $Jv$ are connected by a distinguished edge. Thus, in $\G$ there is a 3-cycle with vertices $u,v, Jv$. Thus, the girth is $3$.

$\rii$ Given $n\geq 3$, consider the basic graph $\B=(n,0,0)$, where the vertices in the $i$-th copy of $\A_1$ are $v_i,v_{n+i}$ with $Jv_i=v_{n+i}$, $1\leq i\leq n$. We expand $\B$ by adding the following complex wedges: $(v_i;v_{i+1},v_{n+i+1})$ for $1\leq i \leq n-1$, and $(v_n; v_1,v_{n+1})$. Then $\G=\widetilde{\B}$ thus obtained carries an adapted complex structure, and there is exactly one $n$-cycle in $\G$, with vertices $v_1,v_2,\ldots,v_n$. Therefore, the girth of $\G$ is $n$. 
\end{proof}

\begin{example}
The graph below is obtained from the construction given in Proposition \ref{prop:girth}$\rii$ for $n=6$:
\begin{center}
\begin{tikzpicture}[scale=1]\footnotesize
\tikzset{every loop/.style={looseness=10}}
\node (1) at (0,1) [draw,circle,inner sep=1pt,fill=black,label=0:{$v_1$}] {};
\node (2) at (-1,0.5) [draw,circle,inner sep=1pt,fill=black,label=90:{$v_2$}] {};
\node (3) at (-1,-0.5) [draw,circle,inner sep=1pt,fill=black,label=180:{$v_3$}] {};
\node (4) at (0,-1) [draw,circle,inner sep=1pt,fill=black,label=180:{$v_4$}] {};
\node (5) at (1,-0.5) [draw,circle,inner sep=1pt,fill=black,label=270:{$v_5$}] {};
\node (6) at (1,0.5) [draw,circle,inner sep=1pt,fill=black,label=0:{$v_6$}] {};
\node (7) at (0,1.7) [draw,circle,inner sep=1pt,fill=black,label=90:{$v_8$}] {};
\node (8) at (-1.7,0.9) [draw,circle,inner sep=1pt,fill=black,label=135:{$v_9$}] {};
\node (9) at (-1.7,-0.9) [draw,circle,inner sep=1pt,fill=black,label=225:{$v_{10}$}] {};
\node (10) at (0,-1.7) [draw,circle,inner sep=1pt,fill=black,label=270:{$v_{11}$}] {};
\node (11) at (1.7,-0.9) [draw,circle,inner sep=1pt,fill=black,label=-15:{$v_{12}$}] {};
\node (12) at (1.7,0.9) [draw,circle,inner sep=1pt,fill=black,label=45:{$v_{7}$}] {};
\path
(1) edge [-,black] node[above=-13] {} (2)
(2) edge [-,black] node[above=-15] {} (3)
(3) edge [-,black] node[above=-1] {} (4)
(4) edge [-,black] node[above=-2] {} (5)
(5) edge [-,black] node[above=-2] {} (6)
(6) edge [-,black] node[above=-13] {} (1)
(7) edge [-,black] node[above=-2] {} (1)
(8) edge [-,black] node[above=-0] {} (2)
(9) edge [-,black] node[above=-2] {} (3)
(10) edge [-,black] node[above=-2] {} (4)
(11) edge [-,black] node[above=-15] {} (5)
(12) edge [-,black] node[above=0] {} (6);
\end{tikzpicture} 
\end{center}
Accordingly, it has a unique cycle, which has length 6. An adapted complex structure on this graph is given by:
\[ Jv_i=v_{6+i} \quad (1\leq i \leq 6), \quad  J e_{i,i+1}=e_{i,7+i} \quad (1\leq i\leq 5), \quad J e_{1,6}=-e_{6,7}. \]
\end{example}

\subsection{Perfect matchings}

Let us recall the following notion from graph theory. Given a graph $\G=\G(V,E)$, a \textit{perfect matching} in $\G$ is a subset $M$ of $E$, such that every vertex in $V$ is adjacent to exactly one edge in $M$. If $\G$ admits a perfect matching $M$ then $|V|$ is even and the number of edges in $M$ is $|V|/2$. 

In the setting of graphs with an adapted complex structure, we have the following result, which establishes the existence of a perfect matching in certain families:

\begin{proposition}\label{prop:matching}
    Let $\G$ be a graph \ with an adapted complex structure $J$, and let $\G_b$ denote the associated basic subgraph. If $\G_b=(0,0,n)$ for some $n\in \N$ then the set of distinguished edges is a perfect matching in $\G$.
\end{proposition}

\begin{proof}
    This is a consequence of the fact that $\G_b$ is a spanning subgraph of $\G$ and that two distinguished edges in $\G$ are never incident. 
\end{proof}

\begin{example}
Consider the basic graph given by $\B=(0,0,2)$, where one copy of $\A_3$ is given by the $v_1,v_2,v_3,v_4$ and edges $e_{1,2}, e_{3,4}$, and the other copy of $\A_3$ is given by the vertices $v_5,v_6,v_7,v_8$ and wedges $e_{5,8}, e_{6,7}$. We expand this graph by adding the complex edges $(v_1; v_3,v_4)$, $(v_2; v_3,v_4)$, $(v_3; v_6,v_7)$, $(v_4; v_6,v_7)$, $(v_4; v_5,v_8)$ and $(v_7; v_5,v_8)$ (see the figure below). The resulting graph $\G=\widetilde{\B}$ admits a complex structure that can be obtained by \eqref{eq: Jtilde}. Note that the distinguished edges are $e_{1,2}, e_{3,4}, e_{5,8}, e_{6,7}$, and they form a perfect matching in $\G$.
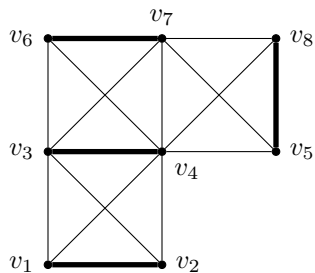
\begin{center}
%    \begin{figure}[ht]
 \begin{tikzpicture}\footnotesize
\tikzset{every loop/.style={looseness=30}}
\node (1) at (0,0) [draw,circle,inner sep=1pt,fill=black,label=180:{$v_1$}] {};
\node (2) at (1.5,0) [draw,circle,inner sep=1pt,fill=black,label=-0:{$v_2$}] {};
\node (3) at (0,1.5) [draw,circle,inner sep=1pt,fill=black,label=180:{$v_3$}] {};
\node (4) at (1.5,1.5) [draw,circle,inner sep=1pt,fill=black,label=315:{$v_4$}] {};
\node (5) at (3,1.5) [draw,circle,inner sep=1pt,fill=black,label=0:{$v_5$}] {};
\node (6) at (0,3) [draw,circle,inner sep=1pt,fill=black,label=180:{$v_6$}] {};
\node (7) at (1.5,3) [draw,circle,inner sep=1pt,fill=black,label=90:{$v_7$}] {};
\node (8) at (3,3) [draw,circle,inner sep=1pt,fill=black,label=0:{$v_8$}] {};
\path
(1) edge [-, line width=2pt, black] node[above=-13] {} (2)
(1) edge [-,black] node[above=-15] {} (3)
(1) edge [-,black] node[above=-1] {} (4)
(2) edge [-,black] node[above=-2] {} (3)
(2) edge [-,black] node[above=-2] {} (4)
(3) edge [-, line width=2pt, black] node[above=-13] {} (4)
(3) edge [-,black] node[above=-2] {} (6)
(3) edge [-,black] node[above=-0] {} (7)
(4) edge [-,black] node[above=-2] {} (5)
(4) edge [-,black] node[above=-2] {} (6)
(4) edge [-,black] node[above=-2] {} (7)
(4) edge [-,black] node[above=-15] {} (8)
(5) edge [-,black] node[above=0] {} (7)
(5) edge [-, line width=2pt, black] node[above=-2] {} (8)
(6) edge [-, line width=2pt, black] node[above=0] {} (7)
(6) edge [-,black] node[above=-15] {} (8);
\end{tikzpicture} 
\end{center}
%\caption{A perfect matching given by distinguished edges}
%    \label{fig:matching}
%\end{figure}
\end{example}

\subsection{Blow-ups of graphs}
Here, we will show that given any non-empty graph we can construct another graph that admits adapted complex structures. We use the notion of blow-up of a graph, which we recall next.

Given a non-empty graph $\G=\G(V,E)$, with $V=\{v_1,\ldots,v_n\}$, consider an integer vector $\mathbf{k}=(k_1,\ldots,k_n)\in \N^n$. The $\mathbf{k}$-blow-up of $\G$, denoted by $\G^{(\mathbf{k})}$, is the graph obtained by replacing every vertex $v_i$ of $\G$ with $k_i$ different vertices $v_i^{(1)},\ldots, v_i^{(k_i)}$, where a vertex $v_i^{(r)}$, $1\leq r\leq k_i$, is adjacent to a vertex $v_j^{(l)}$, $1\leq l\leq k_j$, in the blow-up graph if and only if $i\neq j$ and $v_i$ is adjacent to $v_j$ in $\G$. 

For instance, consider the complete graph $K_2$ and $\mathbf{k}=(k_1,k_2)$. Then, $K_2^{(\mathbf{k})}$ is isomorphic to the complete bipartite graph $K_{k_1,k_2}$.

Regarding the existence of adapted complex structures on blow-up graphs, we have the following result.

\begin{proposition}\label{prop: blow}
Let $\G$ be any graph with vertices $v_1,\ldots, v_n$, and consider an integer vector $\mathbf{k}=(k_1,\ldots,k_n)\in \N^n$ such that $k_i$ is even for all $i$. Then, the $\mathbf{k}$-blow-up $\G^{(\mathbf{k})}$ admits an adapted complex structure such that the corresponding basic subgraph is the union of copies of $\A_1$.
\end{proposition}

\begin{proof}
%Since $k_i$ is even for each $i$, we can group each set of vertices $v_i^{(1)},\ldots, v_i^{(k_i)}$ in pairs, where each such pair determines an $\A_1$ subgraph. 
If $\overline{v_iv_j}$ is an edge in $\G$, then, in $\G^{(\mathbf{k})}$, the vertices $v_i$, $v_j$ and the edge $\overline{v_iv_j}$ are replaced by the complete bipartite graph $K_{k_i,k_j}$. Since $k_i$ and $k_j$ are even, this complete bipartite graph admits an adapted complex structure, thanks to Proposition \ref{prop: bipartite}, such that the corresponding basic subgraph consists only of copies of $\A_1$ (see Example \ref{ex: A1}). Repeating this for any edge in $\G$, we obtain an adapted complex structure on $\G^{(\mathbf{k})}$ such that the associated basic subgraph is the union of copies of $\A_1$.
\end{proof}

\begin{example}
In the figure below, we see the complete graph $K_3$ on the left, and its $\mathbf{k}$-blow-up on the right, where $\mathbf{k}=(2,2,2)$. According to Proposition \ref{prop: blow}, this blow-up graph admits an adapted complex structure $J$ such that the associated basic subgraph is $\G_b=(3,0,0)$. The vertices of each $\A_1$ subgraph are $v_i^{(1)}$ and $v_i^{(2)}$ for $1\leq i\leq 3$.
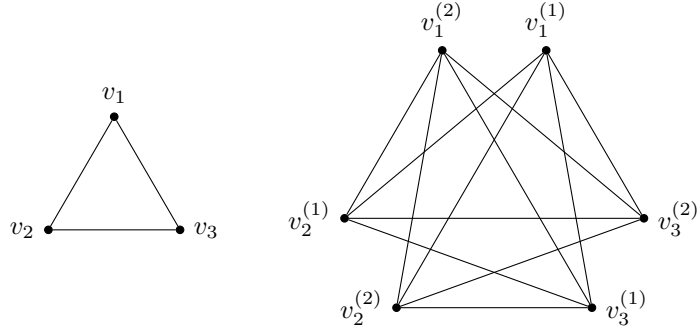
\begin{center}
\begin{tikzpicture}[scale=1]\footnotesize
\tikzset{every loop/.style={looseness=10}}
\node (1) at (0,1) [draw,circle,inner sep=1pt,fill=black,label=90:{$v_1$}] {};
\node (2) at ({cos(210)},{sin(210)}) [draw,circle,inner sep=1pt,fill=black,label=180:{$v_2$}] {};
\node (3) at ({cos(330)},{sin(330)}) [draw,circle,inner sep=1pt,fill=black,label=0:{$v_3$}] {};
\node (4) at ({5+2*cos(310)},{2*sin(310)}) [draw,circle,inner sep=1pt,fill=black,label=0:{$v_3^{(1)}$}] {};
\node (5) at ({5+2*cos(350)},{2*sin(350)}) [draw,circle,inner sep=1pt,fill=black,label=0:{$v_3^{(2)}$}] {};
\node (6) at ({5+2*cos(70)},{2*sin(70)}) [draw,circle,inner sep=1pt,fill=black,label=90:{$v_1^{(1)}$}] {};
\node (7) at ({5+2*cos(110)},{2*sin(110)}) [draw,circle,inner sep=1pt,fill=black,label=90:{$v_1^{(2)}$}] {};
\node (8) at ({5+2*cos(190)},{2*sin(190)}) [draw,circle,inner sep=1pt,fill=black,label=180:{$v_2^{(1)}$}] {};
\node (9) at ({5+2*cos(230)},{2*sin(230)}) [draw,circle,inner sep=1pt,fill=black,label=180:{$v_2^{(2)}$}] {};

\path
(1) edge [-,black] node {} (2)
(1) edge [-,black] node {} (3)
(2) edge [-,black] node {} (3)
(4) edge [-,black] node {} (6)
(4) edge [-,black] node {} (7)
(4) edge [-,black] node {} (8)
(4) edge [-,black] node {} (9)
(5) edge [-,black] node {} (6)
(5) edge [-,black] node {} (7)
(5) edge [-,black] node {} (8)
(5) edge [-,black] node {} (9)
(6) edge [-,black] node {} (8)
(6) edge [-,black] node {} (9)
(7) edge [-,black] node {} (8)
(7) edge [-,black] node {} (9);
\end{tikzpicture} 
\end{center}

\end{example}

\section{Trees and adapted complex structures} \label{sec:trees}

In this section, we use the results from previous sections in order to characterize the forests, namely, graphs with no cycles, that admit an adapted complex structure.

First, we point out that:
\begin{itemize}
    \item[$\diamond$] if $T=T(V,E)$ is a tree then $|V|=|E|+1$, hence the associated Lie algebra $\n_T$ has odd dimension,
    \item[$\diamond$] more generally, if the forest $\G=\G(V,E)$ has $k$ connected components then $|V|=|E|+k$.
\end{itemize}

\begin{theorem}\label{thm:tree}
Let $\G$ be a forest, that is, a graph whose connected components are trees. Then the following conditions are equivalent:
\begin{enumerate}
    \item[$\ri$] $\G$ admits an adapted complex structure,
    \item[$\rii$] $\G$ has an even number of connected components, and each connected component has at most one even vertex.
\end{enumerate}
\end{theorem}

\begin{proof}
$\ri \Rightarrow \rii$ Since $\G$ has an adapted complex structure $J$, the associated Lie algebra $\n_\G$ has even dimension, therefore, the number $k$ of connected components of $\G$ is even. 

If $\G$ is a  basic graph, then it is a forest whose connected components are either isolated vertices or trees isomorphic to the graph $K_2$; therefore, it satisfies condition (ii). 

Assume now that $\G$ is not basic, and let $\G_b$ be the basic subgraph associated with $(\G,J)$, so that $\G$ is an expansion of $\G_b$. The main observation is that if $(u;v,Jv)$ is a complex wedge in $\G$ then there is no edge connecting the vertices $v$ and $Jv$, since otherwise there would be a 3-cycle in $\G$, which is not possible. Thus, the endpoints $v$ and $Jv$ of the complex wedge determine an $\A_1$ subgraph. In particular, this implies that the vertices of $\G$ which are in an $\A_2$ or $\A_3$ subgraph cannot be an endpoint of any complex wedge, they can only be the center of a complex wedge.

Next, we point out that the vertices in an $\A_1$ subgraph can be the endpoints of a complex wedge, but they cannot be the endpoints of two or more complex wedges, since otherwise there would be cycles in $\G$, a contradiction. Thus, a vertex in an $\A_1$ subgraph can be an endpoint of at most one complex wedge, but it could be the center of several complex wedges. In particular, the non-distinguished edges of $\G$ always have at least one endpoint in an $\A_1$ subgraph. 

Consider now an edge $a$ in an $\A_2$ or an $\A_3$ subgraph with vertices $v,\,w$, with $w=Jv$. If $\deg(v)=\deg(w)=1$ then $a$ is a connected component of $\G$ that satisfies condition $\rii$. If, say, $\deg(v)>1$, then $v$ can only be the center of several complex wedges (hence, $\deg(v)$ is odd), and the endpoints of each of these complex wedges determine an $\A_1$ subgraph. Moreover, these endpoints either have degree one or are again centers of other complex wedges whose endpoints are in an $\A_1$ subgraph. This process continues, but eventually it will end. Thus, each of the vertices in the $\A_1$ subgraphs has odd degree. The same happens with the vertex $w$, and therefore the connected component of $\G$ containing the edge $a$ has all its vertices of odd degree, satisfying condition $\rii$. 

Consider now a vertex $v$ of $\G$ which, when considered as a vertex in the basic subgraph $\G_b$, is the isolated vertex of an $\A_2$ subgraph. If $v$ is isolated in $\G$ then it is a connected component of $\G$ satisfying condition $\rii$. If $v$ is not isolated in $\G$ then we know that it can only be the center of several complex wedges (hence, $\deg(v)$ is odd), and the endpoints of each of these complex wedges are in an $\A_1$ subgraph. Moreover, these endpoints either have degree one or are again centers of other complex wedges whose endpoints are in an $\A_1$ subgraph. This process continues, but eventually it will end. Thus, $\deg(v)$ is even (as predicted by Lemma \ref{lem:misc}) 
and each of the vertices in the $\A_1$ subgraphs has odd degree. Therefore, the connected component of $\G$ containing $v$ satisfies condition $\rii$, since $v$ has even degree and all the other vertices (which are in some $\A_1$ subgraph) have odd degree. 

Finally, consider a connected component $T$ of $\G$ which does not contain an edge of an $\A_2$ or $\A_3$ subgraph, and does not contain an isolated vertex of an $\A_2$ subgraph. That is, all the vertices in $T$ are in an $\A_1$ subgraph. We may assume also that $T$ is not an isolated vertex. Let $v_1$ a vertex in $T$ with $\deg(v_1)=1$. Therefore, $v_1$ is an endpoint of a complex wedge $(v_2;v_1,\pm Jv_1)$. 
\begin{itemize}
    \item[$\diamond$] If $v_2$ is not the endpoint of a complex wedge then it could be the center of some other complex wedges, and this implies that $\deg(v_2)$ is even. Moreover, any vertex $w$ connected to $v_2$ has odd degree, since there is the edge connecting $w$ with $v_2$, and $w$ could be the center of several complex wedges. Now, any vertex $x$ connected to $w$ (different from $v_2$) again has odd degree, since there is the edge connecting $x$ with $w$ and $x$ could be the center of several complex wedges. This shows that, in this case, $v_2$ has even degree and all the other vertices have odd degree. Hence, $T$ satisfies condition $\rii$.
    \item[$\diamond$] If $v_2$ is the endpoint of a complex wedge $(v_3;v_2,\pm Jv_2)$ then there are again two possibilities: $v_3$ is not the endpoint of any complex wedge, or it is the endpoint of a complex wedge. 
\end{itemize}
Following this procedure, and since there are no cycles in $T$, we will eventually get a vertex $v_j$ which is not the endpoint of any complex wedge. Reasoning as above, we may see that $v_j$ has even degree, and all the other vertices in $T$ have odd degree. Hence, condition $\rii$ also holds in this case. 

\medskip

$\rii\Rightarrow\ri$ Now assume that the forest $\G$ has an even number of connected components, and that each component has at most one even vertex. Group these components arbitrarily into pairs. Each pair necessarily falls into exactly one of the following cases:
\begin{enumerate}[(a)]
    \item both trees have exactly one even vertex,
    \item one tree has exactly one even vertex while all the vertices in the other tree are odd,
    \item all vertices in both trees of the pair are odd.
\end{enumerate}
We will show, in each case, that the graph formed by each pair of connected components of $\G$ has an spanning subgraph isomorphic to a basic graph, and the other edges in this graph are part of a wedge, so that it can be considered as an expansion of the basic subgraph. In particular, this graph with a pair of connected components has an adapted complex structure, and hence, $\G$ has an adapted complex structure.

For case (a), let $T_1$ and $T_2$ denote the trees in the pair, and let $v_i$ be the only even vertex in $T_i$, $i=1,2$. Consider the empty graph formed by the vertices $v_1$ and $v_2$, which is isomorphic to $\A_1$. If $\deg(v_1)=\deg(v_2)=0$, then we are done, since $T_1\cup T_2=\A_1$ has an adapted complex structure. Assume now that at least one of these vertices has positive degree, say $\deg(v_1)>0$. Then there are vertices $v_{1,1},\ldots,v_{1,2m_1}$ that are adjacent to $v_1$. We can group them in pairs $\{v_{1,2i-1},v_{1,2i}\}$, and each pair determines an $\A_1$ subgraph (there is no edge between the vertices in each pair since otherwise the graph would have a cycle, which is not possible), and $(v_1;v_{1,2i-1},v_{1,2i})$ is a wedge. Each vertex $v_{1,j}$ has odd degree and is already connected to $v_1$, so that it is connected to an even number of vertices different from $v_1$ (it could be zero). If, for instance, $v_{1,j}$ is connected to $v_{1,j,1},\ldots,v_{1,j,2m_2}$, we repeat what we did before: we group them in pairs $\{v_{1,j,2i-1},v_{1,j,2i}\}$, each of these pairs determines an $\A_1$ subgraph and $(v_{1,j};v_{1,j,2i-1},v_{1,j,2i})$ is a wedge. Each vertex $v_{1,j,k}$ has odd degree and is already connected to $v_{1,j}$, so that it is connected to an even number of vertices different from $v_{1,j}$ (it could be zero). We can repeat the process until we arrive at vertices that have all degree equal to one. Thus, we can consider $T_1\cup T_2$ as the expansion of a basic subgraph consisting of several copies of $\A_1$. Re-enumerating arbitrarily the vertices in $T_1\cup T_2$ and using \eqref{eq: Jtilde} we see that $T_1\cup T_2$ admits an adapted complex structure.

The other cases are treated similarly. For case (b), suppose that $T_1$ has a unique even vertex, say $v_1$, and, on the other hand, $T_2$ has only odd vertices. Choose any edge $a$ in $T_2$, with endpoints $x$ and $y$, and consider the subgraph with vertices $v_1,x,y$ and only edge $a=\overline{xy}$. This subgraph is isomorphic to $\A_2$; if this subgraph is the whole graph $T_1\cup T_2$, we are done. Otherwise, if $\deg(x)>0$ or $\deg(y)>0$, then since $x$ and $y$ have odd degree but we have already considered the edge $a=\overline{xy}$, each of these two vertices is connected to an even number of vertices different from $x$ and $y$. Moreover, $v_1$ is also connected to an even number of vertices. Thus, if $\deg(v_1)>0$ we may apply the same procedure as before and see that $T_1\cup T_2$ is obtained from the expansion of the basic subgraph formed by this $\A_2$ and several copies of $\A_1$. Hence, $T_1\cup T_2$ admits an adapted complex structure.

Finally, for case (c), choose any edge $a_1=\overline{x_1y_1}$ in $T_1$ and any edge $a_2=\overline{x_2y_2}$ in $T_2$. The subgraph formed by the vertices $x_1,y_1,x_2,y_2$ and the edges $a_1,\, a_2$ is isomorphic to $\A_3$. If this subgraph is the whole graph $T_1\cup T_2$ then we are done. Otherwise, $x_1,y_1,x_2,y_2$ have odd degree, but since they are already connected with the edges $a_1$ and $a_2$, it follows that each of them is adjacent to an even number of other vertices, and the set of vertices connected to $x_i$ is disjoint with the set of vertices connected to $y_i$ (since there are no cycles in $T_i$), for $i=1,2$. Therefore, we may apply the same procedure as before and see that $T_1\cup T_2$ can be seen as the expansion of the basic subgraph formed by this $\A_3$ and several copies of $\A_1$. Hence, $T_1\cup T_2$ admits an adapted complex structure.
\end{proof}

\begin{remark}
For a forest $\G$ consisting of two connected components $T_1$ and $T_2$, we note that:
\begin{itemize}
\renewcommand{\labelitemi}{$\diamond$}
    \item if both $T_1$ and $T_2$ have only one even vertex, then any adapted complex structure on $\G$ sends, up to sign, the even vertex in $T_1$ to the even vertex in $T_2$;
    \item if $T_1$ has only one even vertex and all the vertices in $T_2$ are odd, then the image of the even vertex by an adapted complex structure may be any edge in $T_2$, which is the only distinguished edge in $\G$;
    \item if all vertices in $T_1$ and $T_2$ are odd, then, for any adapted complex structure $J$ on $\G$, there is exactly one distinguished edge in $T_1$ and one distinguished edge in $T_2$ such that $J$ sends one of them to the other one.
\end{itemize}
\end{remark}

\begin{example}\label{ex:trees}
We exhibit next examples of forests with two connected components $T_1$ and $T_2$ that admit an adapted complex structure. In the examples, the tree $T_1$ has vertices $\{v_i\}$ and edges $\{e_{i,j}\}$, while the tree $T_2$ has vertices $\{w_i\}$ and edges $\{f_{i,j}\}$ .

  \begin{enumerate}[(i)]
      \item Let $T_1$ and $T_2$ be two trees with only one even vertex, as shown below: 
%\begin{figure}[h]
%    \centering
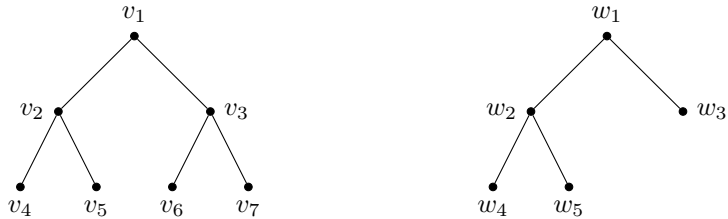
\begin{center}
\begin{tikzpicture}\footnotesize
    \tikzset{every loop/.style={looseness=-10}}
    \node (1) at (0,1) [draw,circle,inner sep=1pt,fill=black,label=90:{$v_1$}] {};
    \node (2) at (-1,0) [draw,circle,inner sep=1pt,fill=black,label=180:{$v_2$}] {};
    \node (3) at (1,0) [draw,circle,inner sep=1pt,fill=black,label=0:{$v_3$}] {};
    \node (4) at (-1.5,-1) [draw,circle,inner sep=1pt,fill=black,label=270:{$v_4$}] {};
    \node (5) at (-0.5,-1) [draw,circle,inner sep=1pt,fill=black,label=270:{$v_5$}] {};
    \node (6) at (0.5,-1) [draw,circle,inner sep=1pt,fill=black,label=270:{$v_6$}] {};
    \node (7) at (1.5,-1) [draw,circle,inner sep=1pt,fill=black,label=270:{$v_7$}] {};
    \path
    (1) edge [-, black] (2)
    (1) edge [-, black] (3)
    (2) edge [-, black] (4)
    (2) edge [-, black] (5)
    (3) edge [-, black] (6)
    (3) edge [-, black] (7);
    \end{tikzpicture} \qquad \qquad \qquad
\begin{tikzpicture}\footnotesize
  \tikzset{every loop/.style={looseness=-10}}
    \node (1) at (0,1) [draw,circle,inner sep=1pt,fill=black,label=90:{$w_1$}] {};
    \node (2) at (-1,0) [draw,circle,inner sep=1pt,fill=black,label=180:{$w_2$}] {};
    \node (3) at (1,0) [draw,circle,inner sep=1pt,fill=black,label=0:{$w_3$}] {};
    \node (4) at (-1.5,-1) [draw,circle,inner sep=1pt,fill=black,label=270:{$w_4$}] {};
    \node (5) at (-0.5,-1) [draw,circle,inner sep=1pt,fill=black,label=270:{$w_5$}] {};
    %\node (6) at (0.5,-1) [draw,circle,inner sep=1pt,fill=black,label=270:{$w_6$}] {};
    %\node (7) at (1.5,-1) [draw,circle,inner sep=1pt,fill=black,label=270:{$w_7$}] {};
    \path
    (1) edge [-, black] (2)
    (1) edge [-, black] (3)
    (2) edge [-, black] (4)
    (2) edge [-, black] (5);
    %(3) edge [-, black] (6)
    %(3) edge [-, black] (7);
    \end{tikzpicture}
\end{center}
%\caption{}\label{fig:trees1}
%\end{figure}   

\noindent In this case, an adapted complex structure $J$ on $T_1\cup T_2$ is given by:
\begin{multicols}{4}
\noindent $Jv_1=w_1$,\\
$Jv_2=v_3$,\\
$Jv_4=v_5$,\\
$Jv_6=v_7$,\\%\columnbreak
$Jw_2=w_3$,\\
$Jw_4=w_5$,\\%\columnbreak
$Je_{1,2}=e_{1,3}$,\\
$Je_{2,4}=e_{2,5}$,\\
$Je_{3,6}=e_{3,7}$,\\%\columnbreak
$Jf_{1,2}=f_{1,3}$,\\
$Jf_{2,4}=f_{2,5}$.\\
%$Jf_{3,6}=f_{3,7}$,\\
\end{multicols}

\item Let $T_1$ be a tree with only one even vertex and $T_2$ a tree whose vertices are all odd, as shown below: 

%\begin{figure}[h]
%\centering
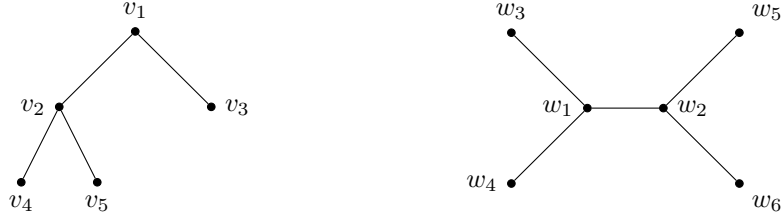
\begin{center}
\begin{tikzpicture}\footnotesize
    \tikzset{every loop/.style={looseness=-10}}
    \node (1) at (0,1) [draw,circle,inner sep=1pt,fill=black,label=90:{$v_1$}] {};
    \node (2) at (-1,0) [draw,circle,inner sep=1pt,fill=black,label=180:{$v_2$}] {};
    \node (3) at (1,0) [draw,circle,inner sep=1pt,fill=black,label=0:{$v_3$}] {};
    \node (4) at (-1.5,-1) [draw,circle,inner sep=1pt,fill=black,label=270:{$v_4$}] {};
    \node (5) at (-0.5,-1) [draw,circle,inner sep=1pt,fill=black,label=270:{$v_5$}] {};
    %\node (6) at (0.5,-1) [draw,circle,inner sep=1pt,fill=black,label=270:{$v_6$}] {};
    %\node (7) at (1.5,-1) [draw,circle,inner sep=1pt,fill=black,label=270:{$v_7$}] {};
    \path
    (1) edge [-, black] (2)
    (1) edge [-, black] (3)
    (2) edge [-, black] (4)
    (2) edge [-, black] (5);
    %(3) edge [-, black] (6)
    %(3) edge [-, black] (7);
    \end{tikzpicture} \qquad \qquad \qquad
\begin{tikzpicture}\footnotesize
  \tikzset{every loop/.style={looseness=-10}}
    \node (1) at (-1,1) [draw,circle,inner sep=1pt,fill=black,label=90:{$w_3$}] {};
    \node (2) at (-1,-1) [draw,circle,inner sep=1pt,fill=black,label=180:{$w_4$}] {};
    \node (3) at (0,0) [draw,circle,inner sep=1pt,fill=black,label=180:{$w_1$}] {};
    \node (4) at (1,0) [draw,circle,inner sep=1pt,fill=black,label=0:{$w_2$}] {};
    \node (5) at (2,1) [draw,circle,inner sep=1pt,fill=black,label=45:{$w_5$}] {};
    \node (6) at (2,-1) [draw,circle,inner sep=1pt,fill=black,label=-45:{$w_6$}] {};
    \path
    (1) edge [-, black] (3)
    (2) edge [-, black] (3)
    (3) edge [-, black] (4)
    (4) edge [-, black] (5)
    (4) edge [-, black] (6);
    \end{tikzpicture}
\end{center}
%\caption{}\label{fig:trees2}
%\end{figure}  

\noindent An adapted complex structure $J$ on $T_1\cup T_2$ is given by:
\begin{multicols}{4}
\noindent $Jv_1=f_{1,2}$,\\
$Jv_2=v_3$,\\
$Jv_4=v_5$,\\
$Jw_1=w_2$,\\
$Jw_3=w_4$,\\
$Jw_5=w_6$,\\
$Je_{1,2}=e_{1,3}$,\\
$Je_{2,4}=e_{2,5}$,\\
%$Je_{3,6}=e_{3,7}$,\\
$Jf_{1,3}=f_{1,4}$,\\
$Jf_{2,5}=f_{2,6}$.
\end{multicols}

\item Let $T_1$ and $T_2$ be trees in which all vertices are odd, as shown below:

%\begin{figure}[h]
%    \centering
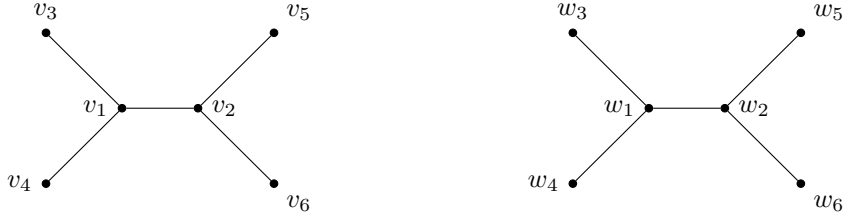
\begin{center}
\begin{tikzpicture}\footnotesize
   \tikzset{every loop/.style={looseness=-10}}
    \node (1) at (-1,1) [draw,circle,inner sep=1pt,fill=black,label=90:{$v_3$}] {};
    \node (2) at (-1,-1) [draw,circle,inner sep=1pt,fill=black,label=180:{$v_4$}] {};
    \node (3) at (0,0) [draw,circle,inner sep=1pt,fill=black,label=180:{$v_1$}] {};
    \node (4) at (1,0) [draw,circle,inner sep=1pt,fill=black,label=0:{$v_2$}] {};
    \node (5) at (2,1) [draw,circle,inner sep=1pt,fill=black,label=45:{$v_5$}] {};
    \node (6) at (2,-1) [draw,circle,inner sep=1pt,fill=black,label=-45:{$v_6$}] {};
    \path
    (1) edge [-, black] (3)
    (2) edge [-, black] (3)
    (3) edge [-, black] (4)
    (4) edge [-, black] (5)
    (4) edge [-, black] (6);
    \end{tikzpicture} \qquad \qquad \qquad
\begin{tikzpicture}\footnotesize
  \tikzset{every loop/.style={looseness=-10}}
    \node (1) at (-1,1) [draw,circle,inner sep=1pt,fill=black,label=90:{$w_3$}] {};
    \node (2) at (-1,-1) [draw,circle,inner sep=1pt,fill=black,label=180:{$w_4$}] {};
    \node (3) at (0,0) [draw,circle,inner sep=1pt,fill=black,label=180:{$w_1$}] {};
    \node (4) at (1,0) [draw,circle,inner sep=1pt,fill=black,label=0:{$w_2$}] {};
    \node (5) at (2,1) [draw,circle,inner sep=1pt,fill=black,label=45:{$w_5$}] {};
    \node (6) at (2,-1) [draw,circle,inner sep=1pt,fill=black,label=-45:{$w_6$}] {};
    \path
    (1) edge [-, black] (3)
    (2) edge [-, black] (3)
    (3) edge [-, black] (4)
    (4) edge [-, black] (5)
    (4) edge [-, black] (6);
    \end{tikzpicture}
\end{center}    
%\caption{}\label{fig:trees3}
%\end{figure} 

\noindent An adapted complex structure $J$ on $T_1\cup T_2$ is given by:
\begin{multicols}{4}
\noindent $Jv_1=v_2$,\\
$Jv_3=v_4$,\\
$Jv_5=v_6$,\\
$Jw_1=w_2$,\\
$Jw_3=w_4$,\\
$Jw_5=w_6$,\\
$Je_{1,2}=f_{1,2}$,\\
$Je_{1,3}=e_{1,4}$,\\
$Je_{2,5}=e_{2,6}$,\\
$Jf_{1,3}=f_{1,4}$,\\
$Jf_{2,5}=f_{2,6}$.\\
\end{multicols}
  \end{enumerate}  
\end{example}

\begin{remark}\label{rem:2trees}
Consider as above a forest $\G$ that is the union of two trees $T_1$ and $T_2$, and assume that $\G$ is equipped with an adapted complex structure $J$. Next, we determine the possible associated basic subgraphs $\G_b$:
\begin{itemize}
\renewcommand{\labelitemi}{$\diamond$}
    \item if there is exactly one even vertex in both $T_1$ and $T_2$ then the associated basic subgraph is $\G_b=\left(\frac{|V|}{2},0,0\right)$, where one of the $\A_1$ subgraphs is formed by the two even vertices;
    \item if $T_1$ has exactly one even vertex and $T_2$ has all its vertices odd then the associated basic subgraph is $\G_b=\left(\frac{|V|-3}{2},1,0\right)$, where the even vertex in $T_1$ is the isolated vertex of the only copy of $\A_2$;
    \item if all vertices in both $T_1$ and $T_2$ are odd then the associated basic subgraph is $\G_b=\left(\frac{|V|-4}{2},0,1\right)$.
\end{itemize}
For the concrete examples appearing in Example \ref{ex:trees}, we have the following:
\begin{itemize}
\renewcommand{\labelitemi}{$\diamond$}
    \item In Case (i), the associated basic subgraph is $\G_b=(6,0,0)$, where the $\A_1$ subgraphs have vertices $\{v_1,w_1\}$, $\{v_2,v_3\}$, $\{v_4,v_5\}$, $\{v_6,v_7\}$, $\{w_2,w_3\}$, $\{w_4,w_5\}$.
    \item In Case (ii), the associated basic subgraph is $\G_b=(5,1,0)$, where the $\A_1$ subgraphs have vertices $\{v_2,v_3\}$, $\{v_4,v_5\}$, $\{v_6,v_7\}$, $\{w_3,w_4\}$, $\{w_5,w_6\}$, and the subgraph $\A_2$ has vertices $\{v_1,w_1,w_2\}$ and edge $f_{1,2}$.
    \item In Case (iii), the associated basic subgraph is $\G_b=(4,0,1)$, where the $\A_1$ subgraphs have vertices $\{v_3,v_4\}$, $\{v_5,v_6\}$, $\{w_3,w_4\}$, $\{w_5,w_6\}$, and the subgraph $\A_3$ has vertices $\{v_1,v_2,w_1,w_2\}$, and edges $\{e_{1,2}, f_{1,2}\}$.
\end{itemize}
\end{remark}

\begin{example}
    Let $P_n$ denote the path graph with $n$ vertices, that is, $V=\{v_1,\ldots,v_n\}$ and $E=\{e_{i,i+1}=\overline{v_iv_{i+1}} \mid i=1,\ldots, n-1\}$. Clearly, $P_n$ is a tree and $\deg(v_1)=\deg(v_n)=1$ and $\deg(v_i)=2$ for $i=2,\ldots, n-1$. Thus, according to Theorem \ref{thm:tree}, the graph $(P_n)_\ast$ admits an adapted complex structure if and only if $n=2$ or $n=3$. 
    
However, we show next that $(P_4)_\ast$ does admit complex structures, necessarily not adapted. Indeed, let $v_1,\ldots, v_5$ the vertices of $(P_4)_\ast$, with $v_5$ the isolated vertex, and let $e_{1,2}:=\overline{v_1v_2}$, $e_{2,3}:=\overline{v_2v_3}$, $e_{3,4}:=\overline{v_3v_4}$ be the edges. Consider the almost complex structure $J$ on $\n_{(P_4)_\ast}$ defined by
\begin{multicols}{4}
\noindent $Jv_1=v_3$, \\
$Jv_3=-v_1$,\\ 
$Jv_2=v_4$, \\
$Jv_4=-v_2$,\\
$Jv_5=e_{1,2}+e_{3,4}$, \\
$Je_{1,2}=-\frac12 (v_5+e_{2,3})$,\\
$Je_{2,3}=e_{1,2}-e_{3,4}$, \\
$Je_{3,4}=-\frac12 (v_5-e_{2,3})$.
\end{multicols}
\noindent Then it is easily verified that $J$ is indeed integrable.
\end{example}

\begin{example}
Assume that $\G$ is a forest with two connected components that admits an adapted complex structure. It follows from Proposition \ref{prop:3-step} that $\G$ carries a $3$-step nilpotent complex structure if and only if all the vertices of one of the trees are odd, and the other tree has exactly one even vertex, which is not isolated.
\end{example}

%\begin{remark}
%Consider the graph $\G=\G(V,E)$ with vertices $V=\{v_1,\ldots,v_5\}$ and edges . Therefore, $\G$ is a forest that does not satisfy the conditions in Theorem \ref{thm:tree}, since $v_2$ and $v_3$ are in the same connected component and they are both even. As a consequence, $\n_\G$ does not admit any adapted complex structure. 
%
%\begin{figure}[h]
%    \centering
%    \begin{tikzpicture}\footnotesize
%\tikzset{every loop/.style={looseness=30}}
%\node (1) at (-1,0) [draw,circle,inner sep=1pt,fill=black!100,label=90:{$v_1$}] {};
%\node (2) at (1,0)  [draw,circle,inner sep=1pt,fill=black!100,label=90:{$v_2$}] {};
%\node (3) at (3,0) [draw,circle,inner sep=1pt,fill=black!100,label=90:{$v_3$}] {};
%\node (4) at (5,0) [draw,circle,inner sep=1pt,fill=black!100,label=90:{$v_4$}] {};
%\node (5) at (7,0)  [draw,circle,inner sep=1pt,fill=black!100,label=90:{$v_5$}] {};
%
%
%\path
%(1) edge [-, black] node[above=-0] {} (2)
%(2) edge [-, black] node[above=-0] {} (3)
%(3) edge [-, black] node[above=-0] {} (4);
%
%\end{tikzpicture}
%  % \caption{Graph $\A_2\cup \A_2$ or $\A_3\cup \A_1$}
%   \label{fig: not adapted}
%\end{figure} 
%
%However, there are complex structures on $\n_\G$. For instance, 
%\end{remark}

\section{Special Hermitian structures on Lie algebras from graphs}\label{sec:metrics}

A Hermitian structure on a smooth manifold $M$ is a pair $(J,g)$ where $J$ is a complex structure and $g$ is a Hermitian metric, that is, a Riemannian metric such that $J$ is $g$-orthogonal or equivalently, $g$-skew-symmetric. We say that $(M,J,g)$ is a Hermitian manifold. 

The fundamental 2-form associated with $(J,g)$ is $\omega(\cdot,\cdot)=g(J\cdot,\cdot)$. When $\omega$ is closed the Hermitian manifold (or the metric $g$) is called \textit{Kähler}. The topology of compact K\"ahler manifolds is well understood: for instance, the odd-indexed Betti numbers of such a manifold are even, and the even-indexed Betti numbers are always positive. The Kähler condition may be too restrictive, and as a consequence, other conditions on the fundamental 2-form which are weaker than being closed have been introduced. 
We recall next some of these non-K\"ahler Hermitian conditions. If $(M,J,g)$ is a Hermitian manifold with $\dim_\C M=n$, then the metric $g$ is called:
\begin{enumerate}
\item \textit{Balanced} if its fundamental form $\omega$ satisfies $d\omega^{n-1}=0$, or equivalently $\delta\omega=0$, where $\delta$ is the codifferential associated to $g$. The study of balanced metrics started with the work of Gauduchon (see \cite{gau-BSMF}, where the denomination \textit{semi-Kähler} was used) and Michelsohn (see \cite{Mic}). They are particularly interesting since they are invariant by modifications \cite{AB96}.

\item \textit{Strong Kähler with torsion} (SKT), also called \textit{pluriclosed}, if its fundamental form $\omega$ satisfies $\partial\bar{\partial}\omega=0$. This condition is equivalent to $dc=0$, where $c$ denotes the torsion $3$-form associated with the Bismut connection $\nabla^b$ determined by $(J,g)$. Recall that $\nabla^b$ is the unique connection on $M$ satisfying $\nabla^bJ=\nabla^bg=0$ and whose torsion $T^b$ is skew-symmetric, that is, $c(X,Y,Z):=g(T^b(X,Y),Z)$ is a $3$-form on $M$. SKT metrics have applications in type II string theory and in 2-dimensional supersymmetric $\sigma$-models \cite{IP,St} and have relations with generalized K\"ahler geometry (see for instance \cite{AG,Gua}).
\end{enumerate}

\medskip

Next, we will study the existence of balanced and SKT metrics on 2-step nilpotent Lie algebras arising from graphs equipped with an adapted complex structure. Such a Hermitian structure on the Lie algebra $\n$ gives rise to a left-invariant Hermitian structure on the corresponding simply connected nilpotent Lie group $N$, which, in turn, induces an invariant Hermitian structure on any associated nilmanifold $\Gamma \backslash N$ (that is, a quotient of $N$ by a co-compact discrete subgroup $\Gamma$); these structures on the nilmanifold are still balanced or SKT, respectively. We point out that there is no need to look for Kähler metrics on nilmanifolds, since it was proved in \cite{BG} that if a nilmanifold $\Gamma \backslash N$ admits a (not necessarily invariant) Kähler structure then $N$ is abelian and the nilmanifold is diffeomorphic to a torus. 

In the realm of nilmanifolds, balanced Hermitian structures have been studied for instance in \cite{AV, LUV-IJM, ST, UV}, while SKT Hermitian metrics have been considered in \cite{AN, EFV, EFV1, FPS}.

\medskip

\subsection{Balanced metrics compatible with adapted complex structures}

If $\g$ is a unimodular Lie algebra equipped with a Hermitian structure $(J,\pint)$ then it is well known that $\pint$ is balanced if and only if
\begin{equation} \label{eq: balanced}
\sum_{i=1}^{2n} [e_i, Je_i]=0, 
\end{equation}
where $\{e_1,\ldots, e_{2n}\}$ is an orthonormal basis of $\g$ (see, for example, \cite{AV}). In particular, this applies to nilpotent Lie algebras.

If $\G$ is a graph and $\n_\G$ is the associated 2-step nilpotent Lie algebra, then condition \eqref{eq: balanced} can be understood as follows:

\begin{proposition}\label{prop: balanced}
Let $\G$ be a graph with vertices $v_1,\ldots, v_n$ that admits an adapted complex structure $J$. If $\pint$ is a Hermitian metric on $\n_\G$ such that $\{v_1,\ldots, v_n\}$ is an orthogonal set, then $\pint$ is balanced if and only if there are no distinguished edges. Equivalently, the basic subgraph $\G_b$ associated with $(\G,J)$ is a union of copies of $\A_1$. In particular, if $\G$ is not the empty graph then the adapted complex structure is not abelian.
\end{proposition}
\begin{comment}
\begin{corollary}
In the hypothesis of Proposition \ref{prop: balanced}, the set of vertices $V$ is $J$-invariant and hence $|V|$ is even.
\end{corollary}

\begin{proof}
This is an easy consequence of Lemma \ref{lem:central}. Indeed, if there is a vertex $w$ such that $Jw$ is an edge, then this edge must be distinguished, which contradicts Proposition \ref{prop: balanced}. Thus, $Jw$ is another vertex.
\end{proof}
\end{comment}

There are many graphs satisfying the conditions in Proposition \ref{prop: balanced}. Analyzing the ones already considered in this article, we obtain the following result.

\begin{corollary}
The following graphs give rise to 2-step nilpotent Lie algebras admitting an adapted complex structure and a compatible balanced Hermitian metric:
\begin{enumerate}
    \item[$\ri$] the complete bipartite graphs $K_{2n,2m}$ from Proposition \ref{prop: bipartite},
    \item[$\rii$] the graphs $F_n$ from Example \ref{ex: peine},
    \item[$\riii$] the graphs of girth $n$ given in Proposition \ref{prop:girth}(ii),
    \item[$\riv$] the blow-ups of any graph, according to Proposition \ref{prop: blow},
    \item[$\mathrm{(v)}$] a forest with an even number of connected components, where each connected component has exactly one even vertex.
\end{enumerate}
\end{corollary}

\subsection{SKT metrics compatible with adapted complex structures}

We recall some facts about SKT metrics on nilpotent Lie algebras. We point out that if a nilpotent Lie algebra admits an SKT structure then the Lie algebra is $2$-step nilpotent, according to \cite{AN}. 

Other useful results are summarized in the following lemma.

\begin{lemma}\label{lem: SKT}
Let $(J,\pint)$ be an SKT Hermitian structure on the 2-step nilpotent Lie algebra $\n$. Then:
\begin{enumerate}
    \item[$\ri$] the center $\z$ of $\n$ is $J$-invariant,
    \item[$\rii$] $x\in \z$ if and only if $[x, Jx]=0$.
\end{enumerate}
\end{lemma}

\begin{proof}
    $\ri$ is proved in \cite[Proposition 3.1]{EFV}, whereas $\rii$ is proved in \cite[Lemma 3.4]{AN}.
\end{proof}

 \medskip

In our case of interest, that is, when the 2-step nilpotent Lie algebra arises from a graph and the complex structure is adapted, Lemma \ref{lem: SKT} translates as follows: 

\begin{lemma}\label{lem: SKT-adapted}
If the Lie algebra $\n_\G$ associated to the graph $\G$ admits an SKT Hermitian structure $(J,\pint)$ with adapted complex structure $J$ and $x$ is a non-isolated vertex of $\G$ then:
\begin{enumerate}
    \item[$\ri$] the image $Jx$ of $x$ is, up to sign, another non-isolated vertex of $\G$,
    \item[$\rii$] there is an edge between $x$ and $Jx$, which is therefore a distinguished edge.
\end{enumerate}
\end{lemma}

We point out that the Bismut torsion 3-form on $(\n_\G, J, \pint)$ can be computed using the following expression (see, for instance, \cite[equation (3.2)]{EFV}), which holds indeed for general Hermitian Lie algebras:
\[ c(x,y,z) = -\langle [Jx , Jy], z\rangle - \langle [Jy, Jz],x \rangle -\langle [Jz, Jx],y \rangle,  \quad x,y,z\in \n_\G.\]
From now on, for simplicity, we assume that \textit{the subspace $\mathfrak{v}$ generated by the non-isolated vertices is the orthogonal complement of the center $\z$ of $\n_\G$}. That is, 
\begin{equation}\label{eq: complement}
    \z^\perp=\mathfrak{v}:=\text{span}\{v\in V\mid \deg(v)>0\}.
\end{equation}
To ensure that $c$ returns non-zero values, two of the arguments must be non-isolated vertices and the third one must be an edge. Assuming that $x,y$ are vertices and $z$ is an edge, we have
\begin{equation}\label{eq: 3form c}
c(x,y,z)=-\langle [Jx,Jy], z\rangle, \qquad x,y\in V, \, z\in E. 
\end{equation} 

Next, we will look for examples of graphs $\G$ with an adapted complex structure $J$ that admits an SKT metric. The graphs we are looking for have to satisfy then the conditions in Lemma \ref{lem: SKT-adapted}.
We begin with low-dimensional examples.

\begin{example}\label{ex: skt dim4}
In dimension 4, there is only one 2-step nilpotent Lie algebra: it is the direct product of the 3-dimensional Heisenberg Lie algebra and $\R$, and it arises from the graph $\A_2$ in Example \ref{ex: first examples}. Recall that there is an adapted complex structure on $\h_3\times\R$ given by $Jv_1=v_2, \, Jv_3=e_{1,2}$. If we consider on $\h_3\times \R$ the inner product $\pint$ making the basis $\{v_1,v_2,v_3,e_{1,2}\}$ orthonormal, then $(J,\pint)$ is a Hermitian structure on $\h_3\times \R$ whose corresponding Bismut torsion 3-form is, according to \eqref{eq: 3form c}, 
\[ c= -v^1\wedge v^2\wedge e^{1,2},\]
where $\{v^1,v^2,v^3,e^{1,2}\}$ denotes the dual basis of $1$-forms. Since $dv^i=0$ for $i=1,2,3$ and $de^{1,2}=-v^1\wedge v^2$, we have $dc=0$ and therefore $(\h_3\times\R,J,\pint)$ is SKT. This fact is already well known (see, for instance, \cite[Theorem 5.1]{MS}).
\end{example}

\begin{example}\label{ex: skt dim6}
The classification of 6-dimensional nilpotent Lie algebras admitting SKT structures was obtained in \cite[Theorem 3.2]{FPS}: there are only 4 (excluding the abelian one), from which only two arise from graphs. Indeed, they are $\n_1:=\h_3\times \R^3$ and $\n_2:=\h_3\times \h_3$. The first arises from the graph $\A_2$ in Example \ref{ex: first examples} adding two more isolated vertices, and it follows from Example \ref{ex: skt dim4} that it admits a product SKT structure. 

As for $\n_2$, it is isomorphic to $\h_3\times \h_3$ and it arises from the graph $\A_3$ in Example \ref{ex: first examples}. 
Recall that there is an adapted complex structure on $\n_2$ given by $Jv_1=v_2, \, Jv_3=v_4, \, Je_{1,2}=e_{3,4}$. If we consider on $\n_2$ the inner product $\pint$ making the basis $\{v_1,v_2,v_3,v_4,e_{1,2},e_{3,4}\}$ orthonormal, then $(J,\pint)$ is a Hermitian structure on $\n_2$ whose corresponding Bismut torsion 3-form is, according to \eqref{eq: 3form c}, 
\[ c= -v^1\wedge v^2\wedge e^{1,2}-v^3\wedge v^4\wedge e^{3,4}.\]
Since $dv^i=0$ for $i=1,2,3,4$, $de^{1,2}=-v^1\wedge v^2$ and $de^{3,4}=-v^3\wedge v^4$, we have $dc=0$ and therefore $(\h_3\times\h_3,J,\pint)$ is SKT. 
\end{example}

\smallskip

In Examples \ref{ex: skt dim4} and \ref{ex: skt dim6} we obtained examples of SKT structures on 2-step nilpotent Lie algebras arising from graphs; moreover, all the examples arise from basic graphs. We will show that this is always the case when we assume, in addition, that the distinguished basis $V\cup E$ of $\n_\G$ is orthonormal (in fact, a milder condition will suffice). 

\begin{proposition}\label{prop: E orthogonal}
Let $\G=\G(V,E)$ be a graph with an adapted complex structure $J$. If $(\n_\G, J)$ admits an SKT metric $\pint$ such that $E$ is an orthogonal set then $\G$ is basic.
\end{proposition}

\begin{proof}
Let $v,w$ be non-isolated vertices of $\G$, with $w\neq \pm Jv$. Since the Hermitian structure is SKT we have $dc=0$, where $c$ the Bismut torsion 3-form. In the proof of \cite[Lemma 3.4]{AN} it was shown that $dc=0$ together with $[\n_\G,\n_\G]\subseteq \z$ and the $J$-invariance of $\z$ imply the following formula:
\begin{equation}\label{eq: romi y marina} 
2\langle [v,Jv],[w,Jw]\rangle =\|[v,w]\|^2+\|[v,Jw]\|^2+\|[Jv,w]\|^2+\|[Jv,Jw]\|^2. 
\end{equation}
Since both $[v,Jv]$ and $[w,Jw]$ are, up to sign, edges of $\G$ we see that the left-hand side of \eqref{eq: romi y marina} vanishes. Thus, all the terms on the right-hand side also vanish, and in particular we obtain $[v,w]=0$. That is, there are no edges connecting $v$ and $w$; hence, all edges are distinguished. It follows from Proposition \ref{prop: basic} that $\G$ is basic.
\end{proof}

\smallskip

Thus, in our search for SKT structures on 2-step nilpotent Lie algebras arising from graphs other than the basic ones, we must require that the set $E$ of edges not be orthogonal. Moreover, the next result follows from the proof of Proposition \ref{prop: E orthogonal}:

\begin{corollary}\label{cor: SKT}
Let $\G=\G(V,E)$ be a graph with an adapted complex structure $J$ such that $(\n_\G, J)$ admits an SKT metric. The following statements hold:
\begin{enumerate}
    \item[$\ri$] if there is a non-distinguished edge joining the vertices $v$ and $w$ then the distinguished edges $\overline{vJv}$ and $\overline{wJw}$ are not orthogonal;
    \item[$\rii$] if $a$ and $b$ are distinguished edges with $b=Ja$ then there is no edge joining an endpoint of $a$ with an endpoint of $b$.
\end{enumerate}
\end{corollary}

Taking into account Lemma \ref{lem: SKT-adapted}, Proposition \ref{prop: E orthogonal} and Corollary \ref{cor: SKT}, we have obtained the following example.

\begin{example}\label{ex:SKT}
Let $\B_n$ a basic graph consisting of one copy of $\A_2$ and $n$ copies of $\A_3$, that is, $\B_n=(0,1,n)$. For $1\leq j\leq n$, we label the vertices and edges of the $j$-th $\A_3$ as follows: the edges are $a_j$ and $b_j$; the edge $a_j$ has vertices $x_j$ and $y_j$, while the edge $b_j$ has vertices $z_j$ and $w_j$. Regarding $\A_2$, the vertices are denoted $v_0, v_1, v_2$, with $v_0$ the isolated vertex, and the edge $\overline{v_1v_2}$ is denoted $t_0$. 

Moreover, we consider on $\B_n$ the following adapted complex structure $J$: 
\[Ja_j=b_j, \quad Jx_j=y_j, \quad Jz_j=w_j, \quad Jv_0=t_0, \quad Jv_1=v_2. \] 
We expand $\B_n$ by adding the edges $f_j:=\overline{v_1x_j}$, $g_j=\overline{v_1y_j}$, $h_j:=\overline{v_2z_j}$ and $k_j:=\overline{v_2w_j}$. The resulting expanded graph $\widetilde{\B_n}$ will be denoted $\G_n$, and the Lie bracket on $\n_{\G_n}$ is given by, for $1\leq j \leq n$:
\begin{gather}\label{eq: corchetes SKT}
[v_1,v_2]=t_0, \quad  [x_j,y_j]=a_j, \quad [z_j,w_j]=b_j, \\ [v_1,x_j]=f_j, \quad [v_1,y_j]=g_j, \quad [v_2,z_j]=h_j, \quad [v_2,w_j]=k_j.\nonumber
\end{gather} 
Note that $\dim \n_{\G_n}=10n+4$, with $\dim \z(\n_{\G_n})=|E\cup \{v_0\}|=6n+2$. If we extend $J$ by setting \[Jf_j=g_j, \quad Jh_j=k_j,\]  
it follows from Theorem \ref{thm: expanded graph} (after an appropriate enumeration of the vertices) that $J$ is an adapted complex structure on $\G_n$. The distinguished edges in $\G_n$ are $\{t_0,a_1,\ldots,a_n,b_1,\ldots, b_n\}$. See Figure \ref{fig: mariposa-n} below for $n=3$ (only the vertices have been labeled). 
\begin{figure}[h]
    \centering
    \begin{tikzpicture}\footnotesize
\node (1) at (-3,3) [draw,circle,inner sep=1pt,fill=black,label=180:{$x_1$}] {};
\node (2) at (-3,2)  [draw,circle,inner sep=1pt,fill=black,label=180:{$y_1$}] {};
\node (3) at (-3,1) [draw,circle,inner sep=1pt,fill=black,label=180:{$x_2$}] {};
\node (4) at (-3,0) [draw,circle,inner sep=1pt,fill=black,label=180:{$y_2$}] {};
\node (5) at (-3,-1) [draw,circle,inner sep=1pt,fill=black,label=180:{$x_3$}] {};
\node (6) at (-3,-2) [draw,circle,inner sep=1pt,fill=black,label=180:{$y_3$}] {};
\node (7) at (3,3) [draw,circle,inner sep=1pt,fill=black,label=0:{$z_1$}] {};
\node (8) at (3,2) [draw,circle,inner sep=1pt,fill=black,label=0:{$w_1$}] {};
\node (9) at (3,1) [draw,circle,inner sep=1pt,fill=black,label=0:{$z_2$}] {};
\node (10) at (3,0) [draw,circle,inner sep=1pt,fill=black,label=0:{$w_2$}] {};
\node (11) at (3,-1) [draw,circle,inner sep=1pt,fill=black,label=0:{$z_3$}] {};
\node (12) at (3,-2) [draw,circle,inner sep=1pt,fill=black,label=0:{$w_3$}] {};
\node (13) at (-1,0.5) [draw,circle,inner sep=1pt,fill=black,label=280:{$v_1$}] {};
\node (14) at (1,0.5) [draw,circle,inner sep=1pt,fill=black,label=260:{$v_2$}] {};
\node (15) at (0,1.5) [draw,circle,inner sep=1pt,fill=black,label=90:{$v_0$}] {};

\path
(13) edge [-, black] node[above=0] {} (1)
(13) edge [-, black] node[above=0] {} (2)
(13) edge [-, black] node[above=0] {} (3)
(13) edge [-, black] node[above=0] {} (4)
(13) edge [-, black] node[above=0] {} (5)
(13) edge [-, black] node[right=0] {} (6)
(13) edge [-, black] node[right=0] {} (14)
(1) edge [-, black] node[right=0] {} (2)
(3) edge [-, black] node[right=0] {} (4)
(5) edge [-, black] node[right=0] {} (6)
(7) edge [-, black] node[right=0] {} (8)
(9) edge [-, black] node[right=0] {} (10)
(11) edge [-, black] node[below=7] {} (12)
(14) edge [-, black] node[above=0] {} (7)
(14) edge [-, black] node[above=0] {} (8)
(14) edge [-, black] node[above=0] {} (9)
(14) edge [-, black] node[above=0] {} (10)
(14) edge [-, black] node[above=0] {} (11)
(14) edge [-, black] node[right=0] {} (12);
\end{tikzpicture}
    \caption{Graph $\G_3$}
    \label{fig: mariposa-n}
\end{figure}

Let us define now an inner product $\pint$ on $\n_{\G_n}$ satisfying the following conditions:
\begin{itemize}
\renewcommand{\labelitemi}{$\diamond$}
    \item $\{v_1,v_2\}\cup \{x_j, y_j, z_j, w_j, f_j, g_j, h_j, k_j\}_{1\leq j\leq n}$ is an orthonormal set,
    \smallskip
    \item $\{v_0,t_0\}\cup \{a_j, b_j\}_{1\leq j \leq n}$ is orthogonal to the set in the previous item, 
    \smallskip
    \item $\langle t_0,a_j\rangle = -\langle v_0,b_j\rangle=1$, $\langle t_0,b_j\rangle = \langle v_0,a_j\rangle=1$ for all $j$, and 
    \smallskip
    \item $\|t_0\|^2=\|v_0\|^2=\|a_j\|^2=\|b_j\|^2=k$ for certain $k>0$, for all $j$. 
\end{itemize}
To show that the inner product $\pint$ on $\n_{\G_n}$ is indeed positive definite we only need to prove that the restriction of $\pint$ to the subspace generated by $\{v_0,t_0\}\cup \{a_j, b_j\}_{1\leq j \leq n}$ is positive definite. This is the content of the next result.

\begin{lemma}
For $k>\sqrt{2n}$, the inner product defined above is positive definite on the subspace generated by $\mathcal{B}:= \{v_0,t_0\}\cup \{a_1, \ldots,a_n\}\cup \{b_1,\ldots,b_n\}$. 
\end{lemma}

\begin{proof}
The matrix of the inner product $\pint$ on the subspace generated by $\mathcal{B}$ is given by the $(2n+2)\times (2n+2)$ matrix
\[ X_n:=\left[
	\begin{array}{cc|ccc|ccc}       
		k & 0 & 1& \cdots&1 &-1 & \cdots & -1\\
		0 & k & 1& \cdots&1 &1 & \cdots & 1\\ \hline 
		1 & 1 & k& & & &  & \\
		\vdots & \vdots & & \ddots& & &  & \\
		1 & 1 & & &k & &  & \\ \hline
        -1 & 1 & & & &k &  & \\
		\vdots & \vdots & & & & & \ddots & \\
		-1 & 1 & & & & &  & k\\
	\end{array}
	\right],\]
where in each empty slot there is a $0$. We will show that this matrix is definite positive if and only if $k>\sqrt{2n}$. It is clear that the block $\begin{bmatrix} k & 0 \\ 0 & k\end{bmatrix}$ is definite positive since $k>0$. Next, let us denote by $A_j$ the submatrix of $X_n$ formed by its first $j+2$ columns and $j+2$ rows, with $1\leq j\leq 2n$. We must show that the determinant of $A_j$ is positive for all $j$; clearly, we have to consider separately the cases $1\leq j\leq n$ and $n+1\leq j\leq 2n$. 

For $1\leq j \leq n$, it is easy to verify that, after performing some elementary row operations, $A_j$ has the same determinant as the $(j+2)\times (j+2)$ matrix 
\[ \left[
	\begin{array}{cc|ccc}       
		k-\frac{j}{k} & -\frac{j}{k} &0& \cdots&0 \\
		-\frac{j}{k} & k -\frac{j}{k} & 0& \cdots&0 \\ \hline 
		1 & 1 & k& &  \\
		\vdots & \vdots & & \ddots&  \\
		1 & 1 & & &k 
	\end{array}
	\right], \]
and therefore 
\begin{equation}\label{eq:detAj-1} 
\det A_j=k^j(k^2-2j), \quad 1\leq j \leq n.
\end{equation}
For $n+1\leq j \leq 2n$, it is easy to verify that, after performing some elementary row operations, $A_j$ has the same determinant as the $(j+2)\times (j+2)$ matrix    
\[ \left[
	\begin{array}{cc|ccc|ccc}       
		k-\frac{j}{k} & -\frac{2n-j}{k} & 0& \cdots&0 &0 & \cdots & 0\\
		-\frac{2n-j}{k} & k -\frac{j}{k}& 0& \cdots&0 &0 & \cdots & 0\\ \hline 
		1 & 1 & k& & & &  & \\
		\vdots & \vdots & & \ddots& & &  & \\
		1 & 1 & & &k & &  & \\ \hline
        -1 & 1 & & & &k &  & \\
		\vdots & \vdots & & & & & \ddots & \\
		-1 & 1 & & & & &  & k\\
	\end{array}
	\right],\]
where the partition of $j+2$ (as rows and columns) is given by $2+n+(j-n)$. Therefore 
\begin{equation}\label{eq:detAj-2}  
\det A_j=k^{j-2}(k^2-2n)(k^2-2(j-n)), \quad n+1\leq j \leq 2n.
\end{equation}
It follows from \eqref{eq:detAj-1} and \eqref{eq:detAj-2} that $\det A_j>0$ if and only if $k^2>2n$. Since $k>0$, the proof is complete. 
\end{proof}

From now on, we always assume $k>\sqrt{2n}$. Then it is clear that $(J,\pint)$ is a Hermitian structure on $\n_{\G_n}$. Note that this Hermitian structure satisfies the conditions of Lemma \ref{lem: SKT-adapted} and Corollary \ref{cor: SKT}, so that $(J, \pint)$ is a candidate to be SKT. We show next that this is indeed the case.

\begin{proposition}
    The Hermitian structure $(J,\pint)$ with $k>\sqrt{2n}$ on $\n_{\G_n}$ is SKT.
\end{proposition}

\begin{proof}
We will show that the Bismut torsion $3$-form $c$ on $\n_{\G_n}$ is closed. First, note that $\pint$ satisfies \eqref{eq: complement}. Thus, we have the following expression for $c$ (given in $\eqref{eq: 3form c}$), where $x$ and $y$ are non-isolated vertices and $z$ is a central element in $\n_{\G_n}$:
\[ c(x,y,z)=-\langle [Jx,Jy], z\rangle. \]
We then have 
\begin{itemize}
\renewcommand{\labelitemi}{$\diamond$}
\begin{multicols}{3}
    \item $c(v_{1},v_{2},t_0)=-k$,
    \item $c(v_{1},v_{2},a_j)=-1$,
    \item $c(v_1,v_2,b_j)=-1$,
    \item $c(x_j,y_j,a_j)=-k$,
    \item $c(x_j,y_j,t_0)=-1$,
    \item $c(x_j,y_j,v_0)=-1$,
    \item $c(z_j,w_j,b_j)=-k$,
    \item $c(z_j,w_j,t_0)=-1$,
    \item $c(z_j,w_j,v_0)=1$,
    \item $c(v_1,z_j,k_j)=-1$,
    \item $c(v_1,w_j,h_j)=1$,
    \item $c(v_2,x_j,g_j)=1$,
    \item $c(v_2,y_j,f_j)=-1$.
\end{multicols}
\end{itemize}
Therefore, we may write
\begin{align} \label{eq: forma c}
    c & = -k v^1\wedge v^2 \wedge t^0 - v^1\wedge v^2 \wedge \sum_j a^j -v^1\wedge v^2 \wedge \sum_j b^j \\ 
    & \quad -k \sum_j x^j\wedge y^j\wedge a^j - \sum_j x^j\wedge y^j\wedge t^0 - \sum_j x^j\wedge y^j\wedge v^0 \nonumber \\
    & \quad -k \sum_j z^j\wedge w^j\wedge b^j - \sum_j z^j\wedge w^j\wedge t^0 + \sum_j z^j\wedge w^j\wedge v^0 \nonumber \\
    & \quad -v^1\wedge \sum_j z^j \wedge k^j + v^1\wedge \sum_j w^j \wedge h^j +v^2\wedge \sum_j x^j \wedge g^j - v^2\wedge \sum_j y^j \wedge f^j
        \nonumber
\end{align}
where, as usual, the superscripts indicate the dual basis. Next, we compute $dc$, using the following relations that are a consequence of \eqref{eq: corchetes SKT}:
\begin{itemize}
\renewcommand{\labelitemi}{$\diamond$}
\begin{multicols}{2}
    \item $dt^0  = -v^1\wedge v^2$, 
    \item $da^j  = -x^j\wedge y^j$,
    \item $db^j  = -z^j\wedge w^j$,
    \item $df^j  = -v^1\wedge x^j$,
    \item $dg^j  = -v^1\wedge y^j$,
    \item $dh^j  = -v^2\wedge z^j$,
    \item $dk^j  = -v^2\wedge w^j$.
\end{multicols}
\end{itemize}
Therefore, we proceed with the calculation of $dc$ as follows:
\begin{align*}
    dc & = v^1\wedge v^2 \wedge \sum_j x^j\wedge y^j + v^1\wedge v^2 \wedge \sum_j z^j\wedge w^j + \sum_j x^j\wedge y^j\wedge v^1 \wedge v^2 \\
    & \quad + \sum_j z^j\wedge w^j\wedge v^1 \wedge v^2 + v^1\wedge \sum_j z^j \wedge v^2\wedge w^j - v^1\wedge \sum_j w^j \wedge v^2 \wedge z^j \\
    & \quad - v^2\wedge \sum_j x^j \wedge v^1 \wedge y^j + v^2\wedge \sum_j y^j \wedge v^1 \wedge x^j \\
    & =0.
\end{align*}
This shows that the Hermitian structure $(J,\pint)$ on $\n_{\G_n}$ is SKT.
\end{proof}
\end{example}

\subsection*{Acknowledgments} This work was partially supported by CONICET, SECyT-UNC and ANPCyT (Argentina).
The authors are grateful to Denis Videla for suggesting the use of the blow-up of graphs, and also to María Alejandra Álvarez, Jonas Deré, Marcos Origlia and Ricardo Podestá for useful comments. The authors would also like to thank the anonymous referee for their valuable suggestions and careful reading of the manuscript.

\

\end{document}